\documentclass[letterpaper, 11pt]{article}

\usepackage{xcolor}

\usepackage[letterpaper, left=2.5cm, top=2cm,right=2.5cm,bottom=2.5cm]{geometry}
\usepackage{psfrag}
\usepackage{amsthm}
\usepackage{hyperref}
\usepackage[onehalfspacing]{setspace}

\usepackage[sfdefault]{FiraSans}
\usepackage[T1]{fontenc}

\usepackage{enumitem}
\setitemize{label=\scriptsize{$\blacksquare$},topsep=0em}
\setenumerate{label=(\alph*), topsep=0em}

\usepackage{amssymb}
\usepackage{amsmath}
\usepackage{epsfig}
\usepackage{graphicx}
\usepackage[all]{xy}
\usepackage{color}
\usepackage{tikz, calc}
\usepackage{pgfplots}
\usepackage{authblk}
\usepackage[T1]{fontenc}
\usepackage[utf8]{inputenc}

\usepackage{helvet}
\usepackage{comment}
\usepackage{marvosym}

\newcommand{\surface}{S}

\newcommand{\R}{\mathbb R}
\newcommand{\Z}{\mathbb Z}
\newcommand{\N}{\mathbb N}

\newcommand{\edot}{\,\cdot\,}
\newcommand{\coloneqq}{:=}

\newcommand{\Kx}{u}
\newcommand{\Kxlow}{u_0}
\newcommand{\Nx}{N_x}
\newcommand{\kx}{\bar{u}}
\newcommand{\ky}{v}
\newcommand{\kylow}{v_0}
\newcommand{\xxi}{\xi}
\newcommand{\xx}{x}
\newcommand{\zz}{z}
\newcommand{\sph}{\mathbb S}
\newcommand{\conv}[1]{\boldsymbol#1}

\DeclareMathOperator{\Ro}{\mathbf{R}}
\DeclareMathOperator{\soft}{soft}

\DeclareMathOperator{\sign}{sign}

\newcommand{\Wo}{\mathbf{W}}
\newcommand{\Ao}{\mathbf{A}}
\newcommand{\Bo}{\mathbf{B}}
\newcommand{\Fo}{\mathbf{F} }
\newcommand{\Do}{\mathbf{D} }
\newcommand{\So}{\mathbf{S} }
\newcommand{\To}{\mathbf{T}}
\newcommand{\Io}{\mathbf{I}}

\newcommand{\herm}{*}
\newcommand{\dd}{ \mathrm{d}  }
\newcommand{\plus}{{\texttt{+}}}
\newcommand{\astx}{\circledast_x}
\newcommand{\astt}{\circledast_t}
 
\newcommand{\low}{{0}}

\newcommand{\X}{\mathbb X}

\newcommand\abs[1]{\left\vert#1\right\vert}
\newcommand\norm[1]{\left\Vert#1\right\Vert}
\newcommand\inner[2]{\left\langle#1,#2\right\rangle}
\newcommand\set[1]{\bigl\{#1\bigr\}}
\newcommand{\enorm}{\left\|\;\cdot\;\right\|}

\newcommand{\La}{\Lambda}
\newcommand{\la}{\lambda}

\newcommand{\data}{y}
\newcommand{\signal}{f}
\newcommand{\kl}[1]{\left(#1\right)}

\usepackage{amsopn}

\newtheorem{theorem}{Theorem}
\newtheorem{alg}[theorem]{Algorithm}
\newtheorem{lemma}[theorem]{Lemma}

\newtheorem{proposition}[theorem]{Proposition}

\newtheorem{corollary}[theorem]{Corollary}

\theoremstyle{definition}

\newtheorem{definition}[theorem]{Definition}

\newtheorem{remark}[theorem]{Remark}

\numberwithin{equation}{section}
\numberwithin{figure}{section}

\title{Multi-Scale Factorization of the Wave Equation with Application to Compressed Sensing Photoacoustic Tomography}

\author{Gerhard Zangerl}
\author{Markus Haltmeier}
\affil{Department of Mathematics, University of Innsbruck\\
Technikestra{\ss}e 13, 6020 Innsbruck, Austria\\
E-mail: {\tt \{gerhard.zangerl,markus.haltmeier\}@uibk.ac.at}}

\begin{document}

\maketitle

\begin{abstract}
Performing a large number of spatial measurements enables high-resolution photoacoustic imaging without specific prior information.  However, the acquisition of spatial measurements is time-consuming, costly, and technically challenging. By exploiting nonlinear prior information, compressed sensing techniques in combination with sophisticated reconstruction algorithms allow reducing the number of measurements while maintaining high spatial resolution. To this end, in this work we propose a multiscale factorization for the wave equation that decomposes the measured data into a low-frequency factor and sparse high-frequency factors. By extending the acoustic reciprocity principle, we transfer sparsity in the measurement domain into spatial sparsity of the initial pressure, which allows the use of sparse reconstruction techniques. Numerical results are presented that demonstrate the feasibility of the proposed framework.


\medskip\noindent \textbf{AMS:}  
45Q05, 65T60, 94A08, 92C55
\end{abstract}

\section{Introduction}

Photoacoustic tomography (PAT) is an emerging imaging technique that combines the high resolution of ultrasound imaging with the high contrast of optical tomography \cite{xu2006photoacoustic}. As illustrated in Figure~\ref{fig:PATPrinciple}, in PAT a semi-transparent sample is illuminated by short pulses of optical energy, which induces an acoustic pressure wave $p \colon \R^3 \times [0, \infty) \to \R $, which depends on spatial position $\xx \in \R^3$ and time $t \geq 0$. The initial pressure distribution $\signal \colon \R^3 \to \R$ is proportional to the internal light absorption characteristics of the sample and provides valuable diagnostic information. Detectors located on a measurement surface $S$ that (partially) surrounds the sample measure the acoustic pressure from which the initial pressure distribution is recovered. Throughout the following, we denote by $\Wo \signal \coloneqq p|_{S \times [0, \infty)}$ the restriction of the acoustic pressure to the measurement surface.

Exact reconstruction formulas for recovering the initial pressure are available for complete data for specific surfaces $S$, see for example, \cite{dreier2020explicit,finch2007spherical,finch2004determining,haltmeier2014universal,kunyansky2011reconstruction,kunyansky2015inversion,nguyen2009family}. Efficient reconstruction schemes in PAT that take into account acoustic attenuation or variable speed of sound have also been developed \cite{Acosta,agranovsky2007uniqueness,ammari2011time,haltmeier2017analysis,huang2013full,kowar2014time,kowar2012photoacoustic,stefanov2009thermo}. Different types of detectors such as linear or circular detectors recording integrals of the acoustic pressure have been investigated  \cite{burgholzer2007temporal,zangerl2009exact}.  In this paper, we consider the constant sound speed case and address the issue of compressed sensing to reduce the amount of data while maintaining high spatial resolution.

\begin{psfrags}
\psfrag{S}{$S$}
\begin{figure}[htb!]
    \center
    \includegraphics[width=\textwidth]{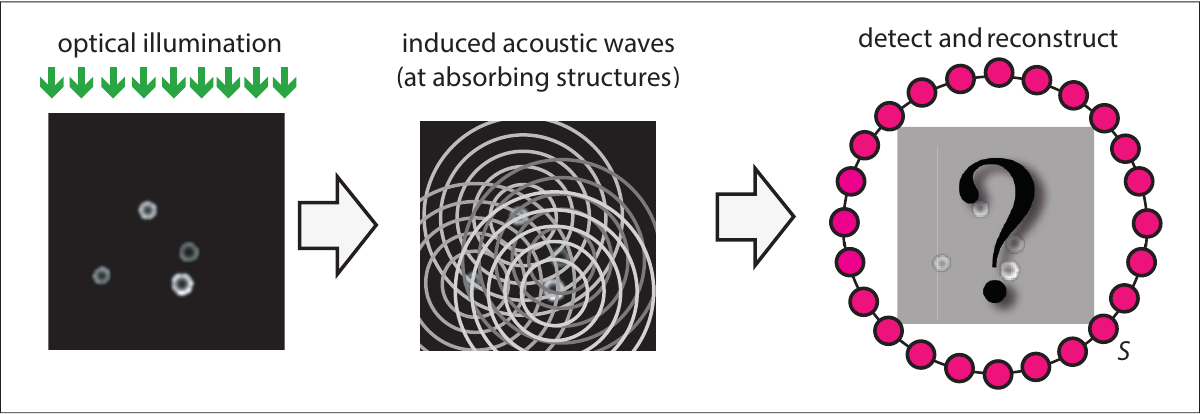}
    \caption{\textsc{Basic principles of PAT.} Left: A sample object is illuminated with short optical pulses. Middle: The optical energy is absorbed in the sample, causing inhomogeneous heating and inducing a subsequent acoustic pressure wave. Right: Acoustic sensors located outside the sample detect the pressure signals, from which an image of the interior is generated. In this work, we develop a specific compressed sensing method that allows to reduce the number of spatial measurements while maintaining a high spatial resolution.} \label{fig:PATPrinciple}
\end{figure}
\end{psfrags}

 \subsection{Compressed sensing PAT}

Acoustic signals in PAT offer a large bandwidth. Therefore, high spatial resolution can be achieved by collecting a sufficiently large number of measurements \cite{haltmeier2016sampling}, as predicted by Shannon's sampling theorem. In practice, however, collecting a large number of spatial measurements requires either a large number of parallel data acquisition channels or a large number of sequential measurements. This either increases the cost and technical complexity of the system or significantly increases measurement time.  Several approaches have been proposed to speed up data acquisition in PAT. For example, a phase contrast method was developed in \cite{nuster2010full} where a reconstruction of the initial pressure can be obtained from projections of the acoustic pressure which can be collected rapidly. In the present work, we use compressed sensing techniques \cite{candes2006stable,donoho2006compressed,foucart2017mathematical,grasmair2011necessary} to reduce the number of measurements in PAT while maintaining high spatial resolution. One of the main challenges in compressed sensing is the development of sophisticated image reconstruction algorithms. In PAT, such CS techniques have been developed in \cite{haltmeier2016compressed,haltmeier2018sparsification,arridge2016accelerated}. Here, we develop a novel image reconstruction strategy based on multiscale factorization of the wave equation that is universally applicable to compressed sensing PAT (CSPAT).

Compressed sensing reconstruction techniques rely on sparsity of the signals to be reconstructed. In PAT, one feasible approach is to express the initial pressure in an appropriate basis. The use of such a strategy leads to a coupled forward model whose solution can be numerically challenging. As an alternative, strategies that apply a time-domain transformation to sparsify PAT data have been proposed in \cite{sandbichler2015novel,haltmeier2016compressed}.  These methods have been demonstrated to reconstruct the high-frequency content of the original pressure very accurately from a significantly reduced set of measurements. However, the proposed differential operators used as sparsifying transforms suppress low-frequency information, resulting in low-frequency artifacts in the reconstruction. In addition to sparsity, the second main component of compressed-sensing reconstruction involves conditions for the subsampled forward matrix that enable linear convergence rates. Necessary and sufficient conditions in a general Hilbert space framework have been derived in \cite{grasmair2011necessary}.  Stable uniform recovery of all sufficiently sparse elements is commonly based on the restricted isometry property \cite{candes2008restricted}.  While in this paper we focus on the sparsity issue, recovery conditions in the context of PAT are briefly discussed in Subsection~\ref{ssec:MSrecon}.

 \subsection{Main Contributions}

In this work, we develop the concept of multiscale time transforms and multiscale factorization for CSPAT. We apply multiscale transforms in the time domain, which split the data into a low-frequency component and several high-frequency components. The basic idea of the proposed reconstruction scheme is to use the acoustic reciprocity principle to show that there is a one-to-one correspondence between the transformed data in the time domain and spatially transformed initial pressure. This factorization allows the use of sparse recovery techniques for the high-frequency part of the initial pressure, while the low-frequency part can be inverted using standard methods.

To be more specific, for a mother wavelet function $\psi \colon \R \to \R$ we set $\ky_j(t) = 2^{j}  \psi(2^j t)$ for $j \geq 1$ and denote by $\ky_0$ a function that contains the missing low-frequency content. This could be, for example, the associated scaling function of the mother wavelet. We then explicitly derive associated functions $\Kx_j \colon \R^2 \to \R$ such that for any initial pressure we have the reciprocal relation
\begin{equation}\label{eq:wavscale}
	\Wo ( \signal \astx \Kx_j )  = \ky_j \astt  (\Wo  \signal )  \quad \text{ for all }  j \in \N  \,.
\end{equation}
The latter identity is then used with compressed sensing data in place of classical data $\Wo \signal$.  Based on \eqref{eq:wavscale}, we develop a reconstruction strategy to recover a multiscale decomposition of the initial pressure consisting of several sparse high-frequency parts and a smooth version of the initial pressure distribution $\signal$. Figure \eqref{fig:sparse} shows a phantom $\signal$, pressure data $\Wo \signal$, and the corresponding multiscale factors $\ky_j \astt (\Wo \signal )$ (top) and $\signal \astx \Kx_j$ (bottom).

 \begin{figure}[htb!]
\includegraphics[width=0.24\textwidth]{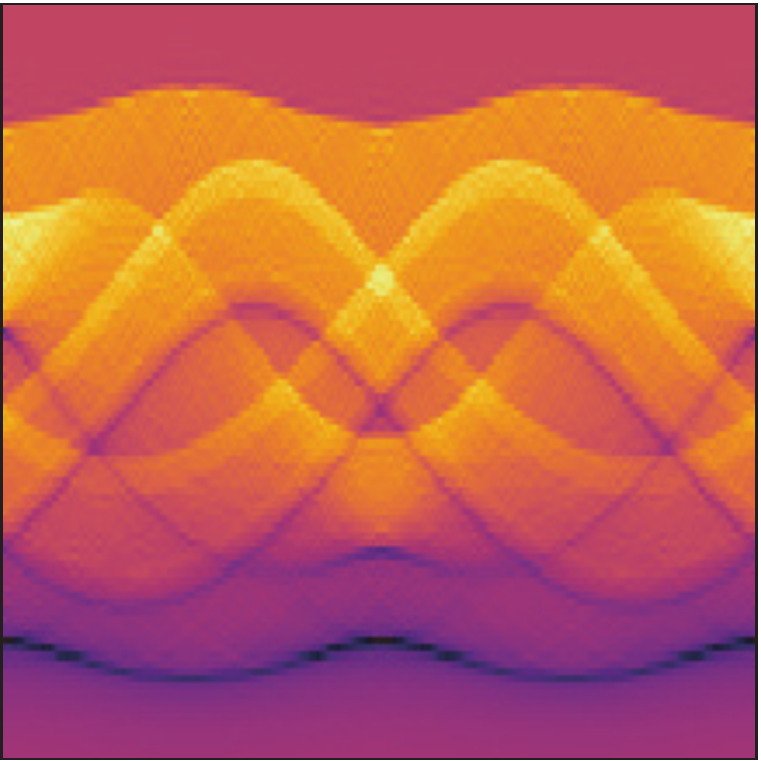}
\includegraphics[width=0.24\textwidth]{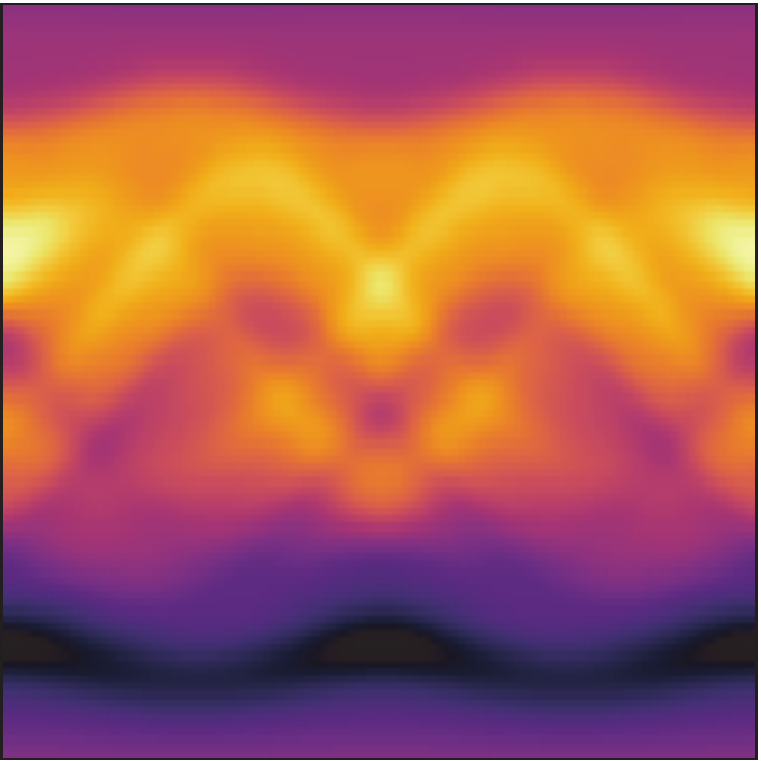}
\includegraphics[width=0.24\textwidth]{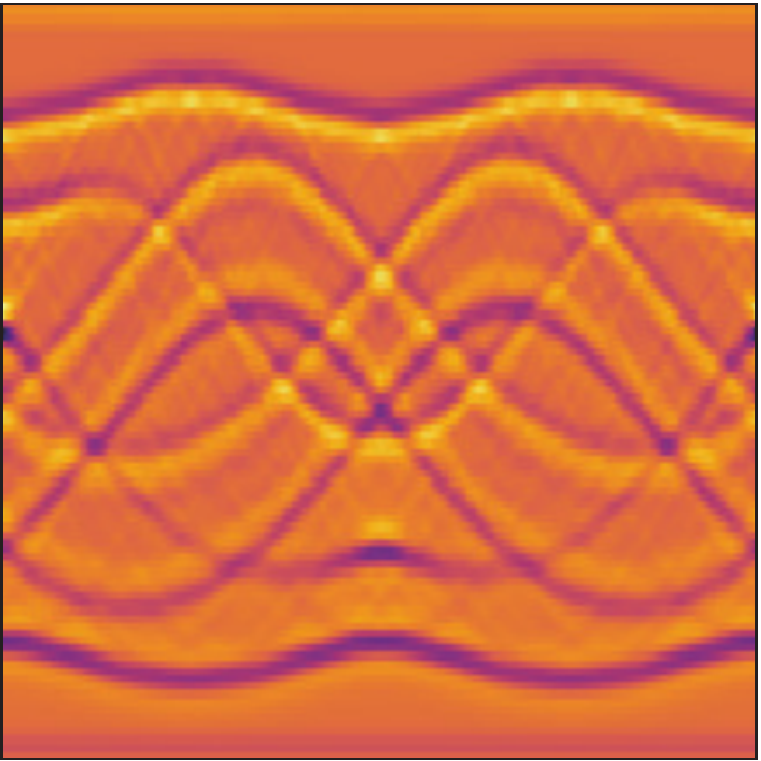}
\includegraphics[width=0.24\textwidth]{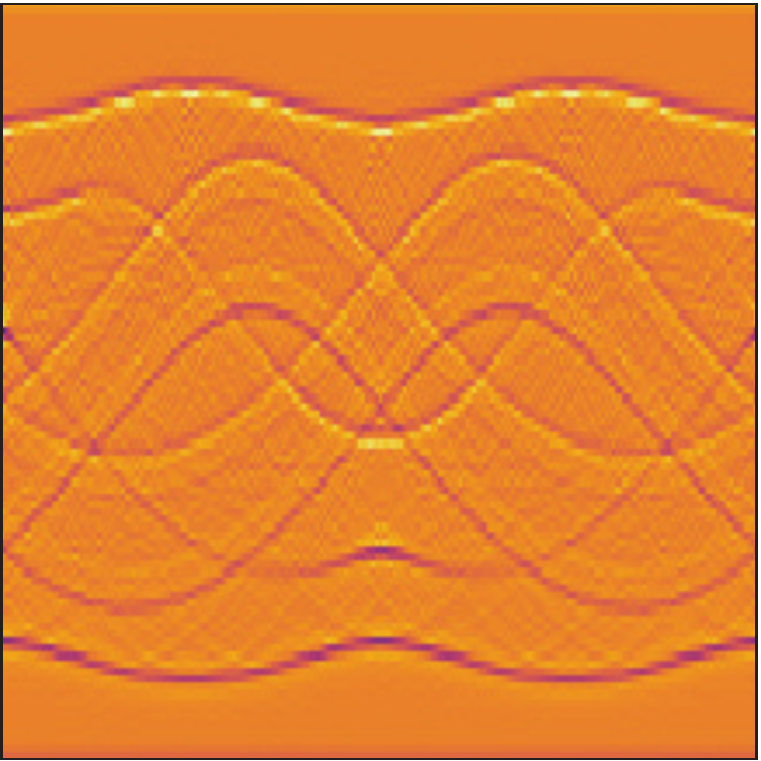}\\[0.1cm]
\includegraphics[width=0.24\textwidth]{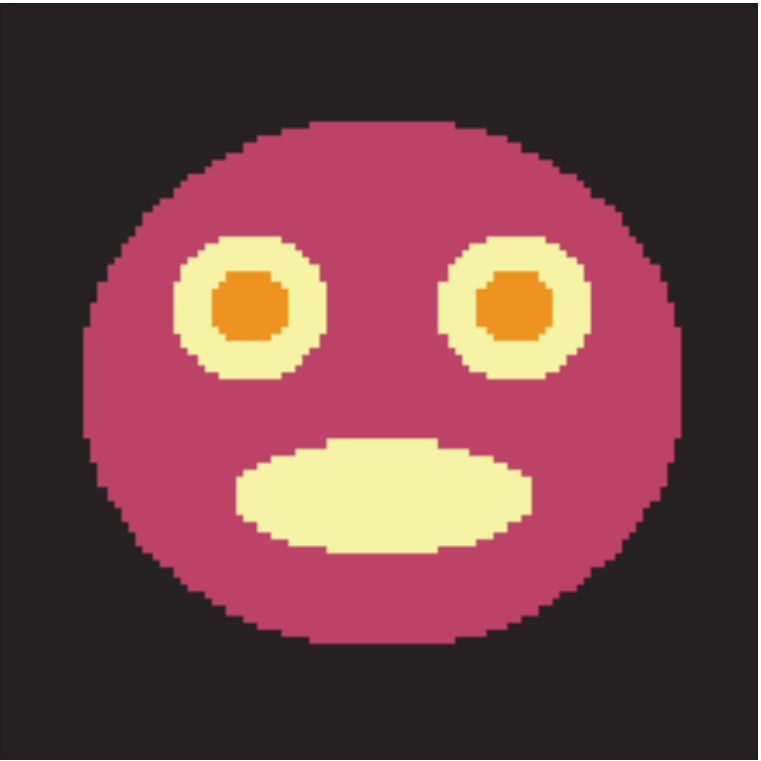}
\includegraphics[width=0.24\textwidth]{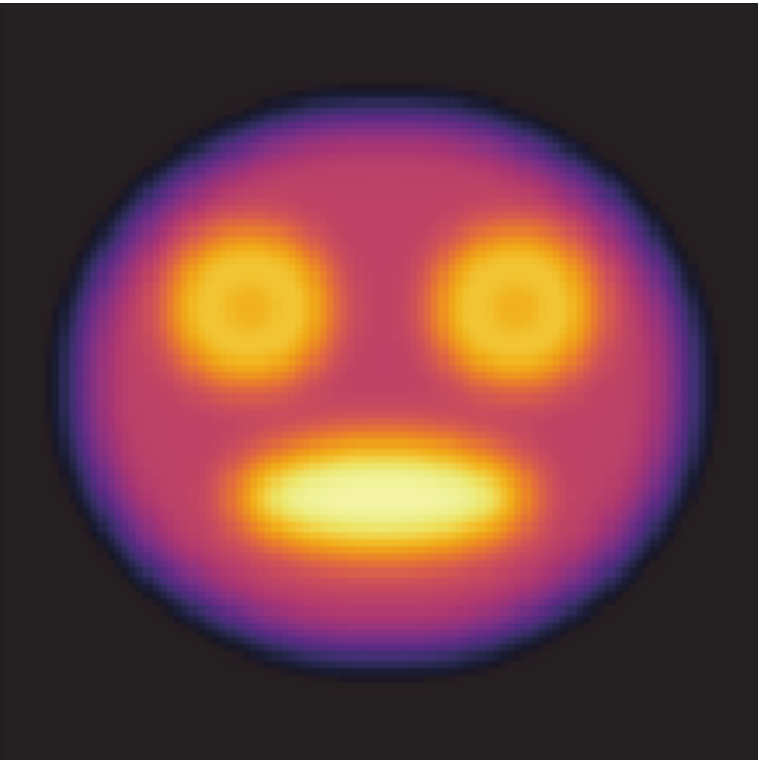}
\includegraphics[width=0.24\textwidth]{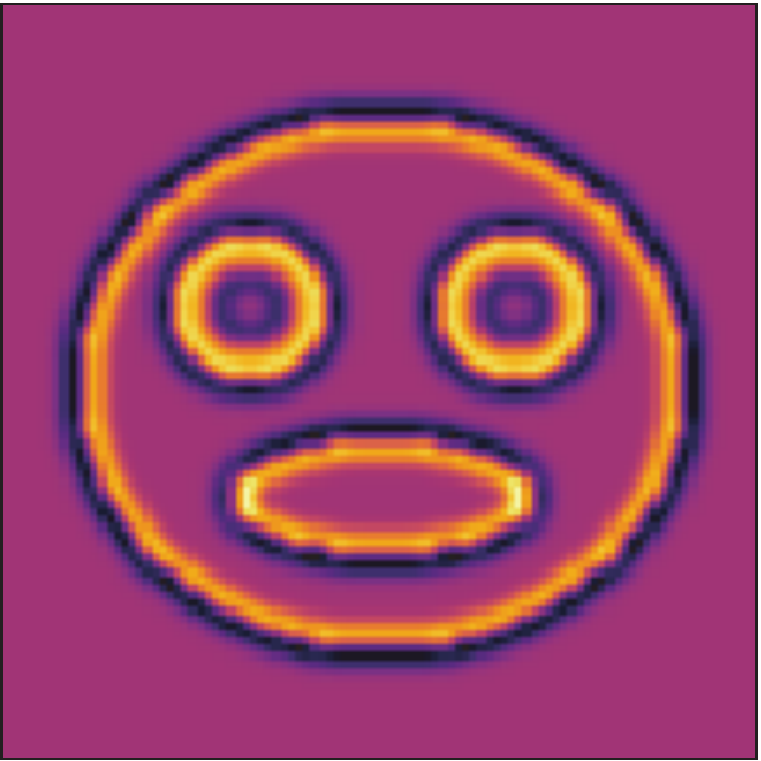}
\includegraphics[width=0.24\textwidth]{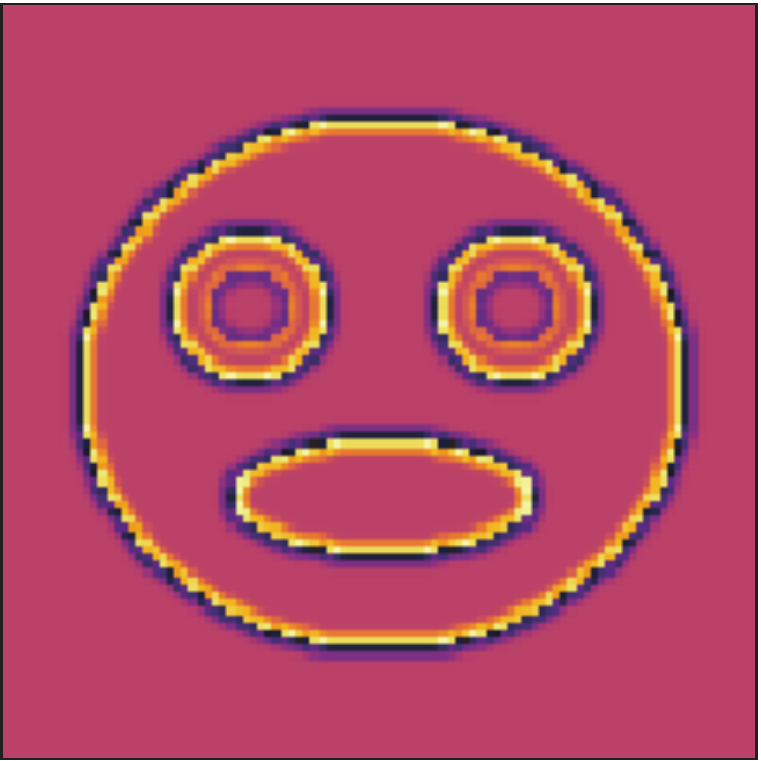}
\caption{Top: Data $\Wo \signal$ (left) and convolved pressure data $\ky_j \astt \Wo \signal$ on the three lowest scales. Bottom: Initial pressure $\signal$ and convolved initial pressure $ \signal \astx \Kx_j$ on the same scales. According to the acoustic reciprocal principle \eqref{eq:wavscale}, the convolved pressure data belong to the convolved initial data for which we have sparsity, and for which we use compressed sensing reconstruction.}
\label{fig:sparse}
\end{figure}

The concept of sparsifying temporal transforms for CSPAT was initially developed in \cite{sandbichler2015novel,haltmeier2016compressed} in two and three spatial dimensions.  These earlier approaches use one-dimensional sparsifying transforms which filter out low-frequency components, thereby leading to low-frequency artifacts in the reconstruction. By considering an additional low-frequency component as well, the proposed multiscale scheme naturally overcomes this drawback of a single multiscale transform. Another solution concerning the missing low-frequency component was proposed in \cite{haltmeier2018sparsification}, where we suggested jointly reconstructing the original pressure and a sparsified version based on the second time derivative. However, the present approach seems more natural and more accessible to a rigorous mathematical analysis.

 \subsection{Outline}

The remainder of the paper is organized as follows. In Section~\ref{sec:prelim} we provide required background of PAT and derive a dual form of the acoustic reciprocal principle.  In Section~\ref{sec:rec} we derive a multiscale factorization for the wave equation. The application to CSPAT is presented in Section~\ref{sec:cspat}. Numerical examples are presented in Section~\ref{sec:num}. The paper ends with a brief summary and outlook in Section~\ref{sec:conclusion}.

\section{Photoacoustic tomography}
\label{sec:prelim}

In this section, we provide the background required from PAT and derive a reciprocal version of the acoustic reciprocity principle that will be useful for our later analysis.

\subsection{Wave equation model}

Throughout, we consider PAT with constant speed of sound. The acoustic pressure is modeled as a function $p \colon \R^d \times [0, \infty) \to \R$, which satisfies the following initial value problem for the wave equation
\begin{alignat}{2}\label{eq:wavee1}
p_{tt} \left( \xx, t \right)  - \Delta p \left( \xx, t \right) &= 0   \quad && ~\text{for}~  (\xx, t )  \in \R^d  \times   \left(0, \infty \right) \,,\\ \label{eq:wavee2}
p    \left( \xx , 0 \right) &=  \signal \left(\xx\right)  \quad   &&~\text{for} ~  \xx  \in \R^d  \,, \\ \label{eq:wavee3}
p_t\left(\xx, 0 \right)     &= 0           \quad              &&~\text{for} ~  \xx  \in \R^d \,.
\end{alignat}
Here $\signal \in C^\infty (\R^d)$ is the initial pressure distribution, which for simplicity we assume to be a smooth function. Notice that in the actual application of PAT the cases $d=2$ and $d=3$ spatial dimensions are relevant \cite{burgholzer2007temporal,finch2007inversion,finch2004determining,kuchment2008mathematics,xu2006photoacoustic}.

Continuous PAT data consist of time-resolved acoustic pressure restricted in space to a smooth detection surface $S \subseteq \R^d$. The continuous domain PAT forward operator is given by
\begin{equation}\label{eq:Wo}
	\Wo \colon C^\infty (\R^3) \to C^\infty ( S \times (0,\infty) ) \colon \signal \mapsto \Wo \signal = p \big|_{S \times (0,\infty)} \,.
\end{equation}
The corresponding complete data inverse problem is to solve the operator equation $\Wo \signal = g$ from possible noisy information. In the last two decades, many methods, including exact reconstruction formulas and iterative methods for different geometries, variable and constant sound speed, and different detector types have been developed; see~\cite{poudel2019survey} for a recent review.
Clearly, only discrete data can be collected in practice. We will present both standard discrete sampling and compressed sensing strategies in Section~\ref{sec:cspat}.

\subsection{Acoustic reciprocal principle}\label{sec:acrep}

Compressed sensing reconstruction techniques are typically based on sparsity of the unknown signals to be reconstructed. To achieve sparsity in PAT, we use the acoustic reciprocity principle in combination with sparse temporal transformations. The acoustic reciprocal principle states that the time-domain manipulation of photoacoustic data corresponds to a spatial convolution of the initial pressure with a radial function. An explicit form of the acoustic reciprocal principle has been proved first for three spatial dimensions in \cite{haltmeier2010spatial} and extended to arbitrary dimensions in \cite{haltmeier2011mollification}.

\begin{proposition}[Acoustic reciprocal principle of \cite{haltmeier2011mollification}]\label{lem:reci}
Let $\Kx \in L^1(\R^d)$ be a compactly supported  radial function of the form $\Kx =  \kx \left( \norm{\edot}  \right)$ and set
\begin{equation}\label{eq:rad}
\Ro \Kx \colon \R \to \R  \colon t \mapsto
\begin{cases}
\kx(t)  & \text{ if }   d = 1 \\
\omega_{d-2}  \int_{|t|}^\infty \kx(s)      \left( s^2 - t^2 \right)^{(d-3)/2}  s \dd s & \text{ if }  d>1 \,,
\end{cases}
\end{equation}
where  $\omega_{d-2}$ denotes the volume of the $(d-2)$-dimensional  unit sphere $\sph^{d-2}$.
Then, for every $\signal \in C^\infty (\R^d)$ we have
\begin{align}\label{eq:reci}
\forall (\xx, t) \in \R^d \times (0, \infty) \colon \quad   \Wo \left(  \Kx \astx \signal   \right)(\xx, t) =  \left( (\Ro \Kx) \astt  \Wo \signal\right)(\xx, t ) \,.
\end{align}
Here  $\astx$ denotes  the spatial convolution in $\R^d$ and  $\astt$ denotes the one-dimensional  convolution applied in the second component.
\end{proposition}

This lemma serves as the basis for the derived multiscale factorization for the wave equation and the resulting sparse reconstruction strategy. In fact, we use the following dual version where we prescribe the temporal filter $\ky$ instead of the spatial filter $\Kx$.

\begin{proposition}[Acoustic  reciprocal  principle, dual version]\label{prop:reciD}
Let  $\ky \colon \R \to \R$  be an even function  with  sufficient decay  such that  $ \ky \circ \sqrt{\abs{\edot}}   \in C^{\lceil (d-1)/2 \rceil}(\R) $ and define
\begin{align}\label{eq:reciD}
	 \Ro^\sharp  \ky  \colon \R^d  \to \R
	 \colon \xx \mapsto
	 \begin{cases}
	   \frac{(-1)^{(d-1)/2}}{\sqrt{\pi^{d-1}}}  \,  \bigl( \kl{\frac{1}{2t}  \frac{\partial}{\partial t}}^{(d-1)/2 }  \ky\bigr) (\norm{\xx} )         & \text{for  $d$  odd}   \\
	   \frac{2 (-1)^{(d-2)/2}}{\sqrt{\pi^d} }      \int_{\norm{\xx} }^\infty  \frac{\kl{\frac{1}{2t}  \frac{\partial}{\partial t}}^{d/2}  \ky(t) }{\sqrt{t^2 - \norm{\xx}^2}} t \dd t        & \text{for $d$  even} \,.
	\end{cases}
\end{align}
Then, for every $\signal \in C^\infty (\R^d)$,
\begin{equation}\label{eq:reci2}
\forall (\xx, t) \in \R^d \times (0, \infty) \colon \quad
\ky  \astt  \left( \Wo \signal  \right)(\xx, t)
=
\Wo \left(   ( \Ro^\sharp  \ky  ) \astx \signal   \right)(\xx, t)    \,.
\end{equation}
\end{proposition}

\begin{proof}
The proof is given in Appendix~\ref{app:reci}.
\end{proof}

Note that the assumption $ \ky \circ \sqrt{\abs{\edot}}   \in C^{\lceil (d-1)/2 \rceil}(\R) $ is made so that the derivatives in \eqref{eq:reciD} are well-defined in the classical sense.

\section{Multiscale factorizations of the wave equation}
\label{sec:rec}

Based on the acoustic reciprocal principle, in this section we derive convolution factorizations for PAT. For that purpose, we fist recall some results for convolutional frames. Then we introduce convolutional frame decompositions in Subsection \ref{sec:decomp}, which are used to derive multiscale factorizations in Subsection~\ref{sec:decomp-multi}.

\subsection{Convolutional  frames}
\label{sec:conv-frames}

 Let $\La$ be an at  most countable index set and consider a family $(\Kx_\la)_{\la \in \La}$ of functions in $L^2(\R^d) \cap L^1(\R^d)$.   According to the convolution theorem we have $ \signal \astx \Kx_\la  =  \Fo_d^{-1} ( (\Fo_d \signal) \cdot (\Fo_d \Kx_\la) )$ for all $\signal \in L^2(\R^d)$. Moreover,  $\signal \astx \Kx_\la$ is well defined  almost everywhere  and satisfies $\signal \astx \Kx_\la  \in L^2(\R^d)$.  Here and in the  following we  denote by  $\Fo_d \signal ( \xxi ) \coloneqq    \int_{\R^d } \signal( \xx) e^{-  i \xx \cdot \xxi} \,  d \xx$ for $\xxi \in \R^d$ the  $d$-dimensional Fourier transform and $\Fo_d^{-1}$ its inverse. We write  $\Kx^\herm(\xx) \coloneqq \Kx(-\xx)$ for  $\Kx \in L^2(\R^d)$ and note  that $\Fo_d \Kx^\herm = \overline{[\Fo_d \Kx]}$, where $\overline{[\edot]}$ denotes complex conjugation.

 \begin{definition}[Convolutional frame] \label{def:frame}
We call a family $\conv{\Kx} = (\Kx_\la)_{\la \in \La} \subseteq (L^2(\R^d) \cap L^1(\R^d))^\La $
a convolutional frame  in $\R^d$, if there are constants $A, B>0$ such that
 \begin{equation} \label{eq:frameTI}
 \signal \in L^2(\R^d)  \colon \quad   A \norm{\signal}^2_2  \leq \sum_{\la \in \la} \norm{\signal \astx \Kx_\la}^2_2
  \leq B \norm{\signal}^2_2 \,.
\end{equation}
If $\conv{\Kx} $  is a convolutional frame, we name $A$, $B$ the frame bounds and call
\begin{enumerate}
\item $\To_{\conv{\Kx}} \colon L^2(\R^d) \to \ell^2( \Lambda, L^2(\R^d)) \colon \signal \mapsto (\Kx_\la \astx \signal)_{\la \in \La}$  analysis operator,
\item  $\To_{\conv{\Kx}}^\herm  \colon \ell^2( \Lambda, L^2(\R^d)) \to L^2(\R^d) \colon (f_\la)_{\la \in \La} \mapsto \sum_{\la \in \La} \Kx_\la^\herm \astx f_\la$ synthesis  operator,

\item $\To_{\conv{\Kx}}^\herm \To_{\conv{\Kx}} \colon L^2(\R^d) \to L^2(\R^d) \colon \signal \mapsto \sum_{\la \in \La} \Kx_\la^\herm \astx \Kx_\la \astx \signal$ frame operator.
\end{enumerate}
Finally, we call $\conv{\Kx}$ tight if $\To_{\conv{\Kx}}^\herm \To_{\conv{\Kx}} = \Io$.
\end{definition}

Note that $\ell^2( \Lambda, L^2(\R^d))$ is a Hilbert space with inner product $\inner{\conv a}{ \conv b}_\La \coloneqq \sum_{\la \in \La}   \inner{a_\la}{b_\la}$ and corresponding norm $\enorm{}_\La$. Using the analysis operator, we can write the defining identity \eqref{eq:frameTI} in the form $A \norm{\signal}^2_2  \leq \norm{\To_{\conv{\Kx}} \signal}_\La^2  \leq B \norm{\signal}^2_2$. Hence the right inequality in \eqref{eq:frameTI} states that $\To_{\conv{\Kx}} \colon L^2(\R^d) \to \ell^2( \Lambda, L^2(\R^d)) $ is well defined an bounded, whereas the left inequality states that $\To_{\conv{\Kx}}$ has a bounded Moore-Penrose inverse $\To_{\conv{\Kx}}^\plus \colon  \ell^2( \Lambda, L^2(\R^d)) \to L^2(\R^d)$. Further note that  $ \To_{\conv{\Kx}}^\herm$ is the adjoint of  $\To_{\conv{\Kx}}$.

\begin{lemma}[Characterization of convolutional frames]  \label{lem:frame}
For any family $\conv{\Kx} = (\Kx_\la)_{\la \in \La}$ of functions in $L^2(\R^d) \cap L^1(\R^d)$, the following statements are equivalent:
\begin{enumerate}[label=(\roman*)]
\item\label{lem:frame1}  $\conv{\Kx} $ is a convolutional frame  with frame bounds $A, B$.
\item\label{lem:frame2}  The identity $ A \leq  \sum_{ \la  \in \La }     \abs{\Fo_d \Kx_\la  }^2  \leq B$ holds  almost everywhere.
\end{enumerate}
\end{lemma}

\begin{proof}
The convolution theorem and the isometry property of Fourier transform imply that \eqref{eq:frameTI} holds if and only  if for all $\signal \in L^2(\R^d)$ we have
\begin{equation*}
	A \int_{\R^d}  \abs{\Fo_d \signal (\xxi)}^2 \, \dd \xi
	\leq  \int_{\R^d} \abs{\Fo_d \signal (\xxi)}^2  \sum_{\la \in \la} \abs{\Fo_d \Kx_\la (\xxi)}^2 \, \dd \xi \leq B \int_{\R^d}  \abs{\Fo_d \signal (\xxi)}^2 \, \dd \xi \,.
\end{equation*}
This, in turn, is equivalent to the fact that Item~\ref{lem:frame2} holds.
\end{proof}

In particular, $\conv{\Kx} $ is tight if and only if $\sum_{ \la  \in \La } \abs{\Fo_d \Kx_\la  }^2 = 1$ holds almost everywhere.

\begin{lemma}[and definition of a dual convolutional frame] \label{lem:dual}
Let  $\conv{\Kx} = (\Kx_\la)_{\la \in \La}$ and  $\conv{w} = (w_\la)_{\la \in \La}$ be two convolutional frames in $\R^d$. The following statements  are equivalent:
\begin{enumerate}[label=(\roman*)]
\item\label{lem:dual1}
The identity $\sum_{\la \in \La} \overline{[\Fo_d w_\la]} \cdot  (\Fo_d \Kx_\la) = 1$ holds almost everywhere.
\item\label{lem:dual2}
The reproducing formula  $ \forall \signal \in L^2(\R^d) \colon
	\To_{\conv{w}}^\herm \To_{\conv{\Kx}} \signal =
	\sum_{\la \in \La} w_\la^\herm \astx (\Kx_\la \astx \signal)
	= \signal$ holds.
\end{enumerate}
If \ref{lem:dual1} and \ref{lem:dual2} hold, we call $(w_\la)_{\la \in \La}$ a dual convolutional frame to $(\Kx_\la)_{\la \in \La}$.
\end{lemma}

\begin{proof}
The linearity and continuity of the Fourier transform together with the convolution theorem show  that  \ref{lem:dual2} is equivalent to $  (\Fo_d \signal) \cdot  \sum_{\la \in \La} (\Fo_d w_\la^\herm) \cdot (\Fo_d \Kx_\la)  = \Fo_d \signal $ for all $  \signal \in L^2(\R^d)$.  Because $\Fo_d w_\la^\herm = \overline{[\Fo_d w_\la]}$ this implies the desired equivalence.
\end{proof}

In particular, Item~\ref{lem:dual1} in Lemma \ref{lem:dual} is satisfied if $\conv{w}$ is taken as the canonical dual convolutional frame $\conv{\Kx}^\plus \coloneqq (\Kx_\la^\plus)_{\la \in \La}$ which is defined by
\begin{equation}\label{eq:cf-dual}
\forall \la \in \La \colon \quad \Fo_d \Kx_\la^\plus  \coloneqq \frac{\Fo_d u_\la  }{  \sum_{\mu \in \La} \abs{\Fo_d u_\mu}^2} \,.
\end{equation}
In this case, $\To_{\conv{\Kx}^\plus} = \To_{\conv{\Kx}}^\plus$ is the Moore-Penrose inverse of the analysis operator $\To_{\conv{\Kx}}$.

\subsection{Convolution factorization}
\label{sec:decomp}

The following concept is central for this paper.

\begin{definition}[Convolution factorization of the wave equation] \label{def:cdf}
Let $\La$ be an at most countable index set and consider families $\conv{\Kx} = (\Kx_\la)_{\la \in \La} \in (L^2(\R^d))^\La$ and $ \conv{\ky} = (\ky_\la)_{\la \in \La} \in (L^2(\R))^\La$. We call the pair $( \conv{\Kx} ,  \conv{\ky})$  a convolution factorization  for $\Wo$ if the  following hold:

\begin{enumerate}[label=(CFD\arabic*), leftmargin=5em]
\item \label{cdf1} $(\Kx_\la)_{\la \in \La}$ is a convolutional frame of $L^2(\R^d)$,
\item \label{cdf2} $(\ky_\la)_{\la \in \La} $ is a convolutional frame of $L^2(\R)$,
\item \label{cdf3} $\forall \signal \in C^\infty_c (\R^d)  \colon  \Wo ( \Kx_\la \astx  \signal) = \ky_\la \astt (\Wo   \signal )$.
\end{enumerate}
\end{definition}

Given a convolution factorization $( \conv{\Kx} , \conv{\ky})$ for $\Wo$ and data $g = \Wo \signal $, the commutation relation~\ref{cdf3} shows that it is sufficient to solve each equation $ \Wo \signal_\la = \ky_\la \astt g$. These equations now involve the unknowns $f_\la = \Kx_\la \astx \signal$ containing specific prior information that we can exploit for inversion. Moreover, we will later show that the same identity holds for any spatial sampling scheme, allowing its application to CSPAT.  In \cite{agranovsky1996approximation} it is shown that $\Wo$ is injective when restricted to $L^p(\R^d)$ with $p \leq 2d/(d-1)$. Thus, if $( \conv{\Kx} , \conv{\ky})$ is a convolution factorization and $\Kx_\la \astx \signal$ has sufficient decay, then we have the reproduction formula
\begin{equation} \label{eq:fac-inv}
	\signal =
	\sum_{\la \in \La} \Kx_\la^\herm  \astx \Wo^{-1} (\ky_\la \astt (\Wo   \signal ) )
\end{equation}
Indeed, the factorization identity \eqref{eq:fac-inv} is the reason why we call a pair $( \conv{\Kx} , \conv{\ky})$ satisfying \ref{cdf1}-\ref{cdf3} a convolutional factorization.

As the  main theoretical result, in this paper we  construct  explicit convolution factorizations for the PAT forward operator. For that purpose, recall
 \begin{equation*}
\Ro^\sharp  \ky  \colon \R^d  \to \R
	 \colon \xx \mapsto
	 \begin{cases}
	   \frac{(-1)^{(d-1)/2}}{\sqrt{\pi^{d-1}}}  \,  \bigl( \kl{\frac{1}{2t}  \frac{\partial}{\partial t}}^{(d-1)/2 }  \ky\bigr) (\norm{\xx} )         & \text{for  $d$  odd}   \\
	   \frac{2 (-1)^{(d-2)/2}}{\sqrt{\pi^d} }      \int_{\norm{\xx} }^\infty  \frac{\kl{\frac{1}{2t}  \frac{\partial}{\partial t}}^{d/2}  \ky(t) }{\sqrt{t^2 - \norm{\xx}^2}} t \dd t        & \text{for $d$  even} \,,
	\end{cases}
\end{equation*}
and  the dual version of the acoustical reciprocal principle
$\ky  \astt  \left( \Wo \signal  \right)
= \Wo \left(   ( \Ro^\sharp  \ky  ) \astx \signal   \right)$
 stated in Proposition~\ref{prop:reciD}.

\begin{theorem}[Construction of convolution factorizations] \label{thm:cfd}
Let $(\ky_\la)_{\la \in \La} \in L^2(\R)^\La$ be a convolutional frame  consisting of even functions $\ky_\la$ with sufficient decay such that such that  $ \ky_\la \circ \sqrt{\abs{\edot}}   \in C^{\lceil (d-1)/2 \rceil}(\R) $ and let $ (\ky_\la^\plus)_{\la \in \La}$  be its canonical dual and set $\Kx_\la \coloneqq  \Ro^\sharp \ky_\la$.
\begin{enumerate}
\item  \label{thm:cfd1}
The pair  $( (\Kx_\la)_{\la \in \La} ,   (\ky_\la)_{\la \in \La} )$ is a convolutional frame decomposition for  $\Wo$.
\item \label{thm:cfd2}
The   canonical dual of $ (\Kx_\la)_{\la \in \La}$ is given by $(\Ro^\sharp \ky_\la^\plus)_{\la \in \La}$.
\item \label{thm:cfd3}
For all $\signal \in C^\infty_c (\R^d)$, the factors  $f_\la = \Kx_\la \astx \signal$ satisfy
\begin{align} \label{eq:rep1}
	\signal &= \sum_{\la \in \La}   \Kx_\la^\plus  \ast
	f_\la \,,
	\\ \label{eq:rep2}
	\Wo  f_\la &=    \ky_\la \astt \Wo \signal  \,.
\end{align}
Hence any  $\signal \in C^\infty_c (\R^d)$  can be recovered from data $\Wo \signal$ by first solving equation \eqref{eq:rep2} for $f_\la$ and then evaluating the  series \eqref{eq:rep1}.
\end{enumerate}
\end{theorem}

\begin{proof}
To show  Item \ref{thm:cfd1} we verify  \ref{cdf1}-\ref{cdf3} from Definition~\ref{def:cdf}.
Item  \ref{cdf1} is satisfied because  $(\ky_\la)_{\la \in \La} \in L^2(\R)^\La$ is  a convolutional frame according to the made assumptions. Item \ref{cdf3} follows from the acoustic reciprocal principle Proposition~\ref{lem:reci}. It remains to verify Item~\ref{cdf2}, namely that the family  $(\Kx_\la)_{\la \in \La} \in L^2(\R)^\La$ is a convolutional frame.  For that purpose, recall that
$\Ro \Ro^\sharp  \ky_\la = \ky_\la$ where $\Ro$ denotes the Radon  transform of a radial function.    According to the Fourier slice theorem  we have  $\Fo_d \Ro^\sharp  \ky_\la = \Fo_1 \ky_\la$. Therefore
$\sum_{\la \in \La} \lvert \Fo_d \Ro^\sharp \ky_\la \rvert^2 =  \sum_{\la \in \La}  \lvert \Fo_1 \ky_\la \rvert^2$ which implies   that $(\Ro^\sharp \ky_\la)_{\la \in \La}$ is a convolutional frame according to Lemma~\ref{lem:frame}. Moreover, we have $\Fo_1 \ky_\la^\plus =   \Fo_1 \ky_\la/  \sum_{\la \in \La} \abs{\Fo_1 \ky_\la}^2$ and therefore
$$
\Fo_d \Ro^\sharp  \ky_\la^\plus
=
\Fo_1 \ky_\la^\plus
= \frac{\Fo_1 \ky_\la}{ \sum_{\la \in \La} \abs{\Fo_1 \ky_\la}^2}
= \frac{\Fo_d \Ro^\sharp  \ky_\la}{  \sum_{\la \in \La} \abs{\Fo_d   \Ro^\sharp  \ky_\la}^2 }\,,
$$
which shows that $ (\Ro^\sharp \ky_\la^\plus)_{\la \in \La}$ is the canonical dual of $ (\Ro^\sharp \ky_\la)_{\la \in \La}$ which is Item \ref{thm:cfd2}. Finally, Item~\ref{thm:cfd3} follows  Items~\ref{thm:cfd1}, \ref{thm:cfd2} and the definitions of a CDF and a dual frame.
\end{proof}

\begin{remark}\label{rem:convergence}
In particular, Theorem~\ref{thm:cfd} states that $(\Kx_\la^\plus)_{\la \in \La}$ is a convolutional frame, which implies that  $\sum_{\la \in \La} \norm{\Kx_\la^\plus \ast f_\la}^2$ is finite. This in turn shows that the series in \eqref{eq:rep1} is absolutely convergent in $L^2(\R^d)$. Besides the condition that $(\ky_\la)_{\la \in \La} $ is a convolutional frame, this only requires that $\ky_\la $ is contained in $C^{\lceil (d-1)/2 \rceil}(\R)$ and has sufficient decay for every $\la \in \La$. These conditions will be fulfilled for the particular choices made in the numerical results.
\end{remark}

\subsection{Multiscale factorization}
\label{sec:decomp-multi}

As shown in the previous subsection, a convolution factorization decomposes the original image reconstruction problem into multiple reconstruction problems, one for every convolved initial pressure $\Kx_\la \astx \signal$.The next basic idea is to take $(\Kx_\la)_{\la \in \La}$ as a multiscale system to be able to take sparsity into account.

For given  $\Kx \in L^2(\R^d) \cap L^1(\R^d)$  consider the scaled versions
\begin{equation}\label{eq:wave1}
  	\Kx_{j} \colon \R^d \to \R \colon   \xx   \mapsto   2^{jd}
	\, \Kx \bigl( 2^j x \bigr)  \quad  \text{ for } j \geq  1 \,.
\end{equation}
According to the scaling property of the $d$-dimensional Fourier transform we have $\Fo_d \Kx_{j} (\xi)=  \Fo_d \Kx (2^{-j} \xi )$. Assume that  $\Fo_d \Kx$ has essential support $\set{\xi \in \R^d \mid  b_0 \leq \norm{\xi} \leq  2 b_0} $, where $b_0 > 0$ is the essential bandwidth. Then, the Fourier transform $\Fo_d \Kx_j$ has essential support $\set{\xi \in \R^d \mid  2^{j-1} b_0  \leq  \norm{\xi} \leq 2^j b_0} $. The union over all $j \geq 1$ covers all frequencies except the low frequencies contained in the ball $B_\low = \set{\xi \in \R^d \mid \norm{\xi}  < b_0}$. In order to obtain a convolutional frame with reasonable constants we therefore  select  another function $\Kxlow \in L^2(\R^d)$ such that $\Fo_d \Kxlow$ covers frequencies in $b_0$.

\begin{definition}[Multiscale convolution decomposition]\label{def:MSD}
Let $\Kxlow, \Kx \in L^2(\R^d) \cap L^1(\R^d)$ and define $\Kx_{j}$ for $j \geq 1$ by \eqref{eq:wave1}. We call the family $ (\Kx_j) _{j \in \N}$ a multiscale convolution decomposition  in $L^2(\R^d)$ if it forms a convolutional frame. For $\signal \in L^2( \R^d)$, we  refer to $  \Kxlow \astx \signal $ as the low-frequency factor  and  to $\Kx_j  \astx \signal $ for $j \geq 1$ as the high-frequency factor at scale $j$.
\end{definition}

According to Lemma \ref{lem:frame} and the scaling property of the Fourier transform, the family $ \conv{\Kx} = (\Kx_j) _{j \in \N}$ is a multiscale decomposition if and only if there are constants $A, B  >0$ such that
\begin{equation} \label{eq:mdsF}
A \leq  \sum_{j \in \N}  \abs{\Fo_d \Kx_j (\xi) }^2 \leq B \quad \text{ for almost every $\xi \in \R^d$} \,.
\end{equation}
Moreover, the canonical dual frame  $\conv{\Kx}^\plus = (\Kx_{j}^\plus) _{j \in \N} $ of  $\conv{\Kx}$   is given by   the Fourier representation  $ \Fo_d \Kx_j^\plus  \coloneqq \Fo_d u_j  /  \sum_{k \in \N} \abs{\Fo_d u_k}^2$ for $j \in \N$. In the one-dimensional case, we denote a multiscale decomposition by $(\ky_{j}) _{j \in \N}$.

 \begin{definition}[Multiscale factorization of the wave equation]
We call a pair $( \conv{\Kx} ,  \conv{\ky})$ a multiscale factorization of $\Wo$ if $( \conv{\Kx} ,  \conv{\ky})$ is  a convolutional frame decomposition such that  $ \conv{\Kx} =  (\Kx_j) _{j \in \N}$  and $ \conv{\ky} =  (\ky_j) _{j \in \N} $ are  multiscale decompositions in $L^2(\R^d)$ and $L^2(\R)$, respectively.
\end{definition}

From Theorem \ref{thm:cfd} we immediately get the following.

\begin{theorem}[Construction of multiscale factorizations] \label{thm:imd}
Let $ \conv{\ky} =(\ky_{j}) _{j \in \N} $ be a multiscale decomposition in $L^2(\R)$ consisting of even functions with sufficient decay, let $ \conv{\ky}^\plus$  be its canonical dual and define $ \conv{\Kx} \coloneqq  (\Ro^\sharp \ky_j) _{j \in \N}$. Then the following holds
\begin{enumerate}

\item  \label{thm:imd1}
The pair  $(  \conv{\Kx},   \conv{\ky} )$ is a multiscale factorization for  $\Wo$.

\item  \label{thm:imd2} The canonical dual $  (\Kx_{j}^\plus) _{j \in \N}$ of $\conv{\Kx}$ has the Fourier representation
\begin{align} \label{eq:dual}
  \Fo_d \Kx_j^\plus(\xi)
&= \frac{ \Fo_1 \ky_j(\norm{\xi}) }{\sum_{j \in \N}   \abs{\Fo_1 \ky_{j} (\norm{\xi}) }^2}
  \,.
\end{align}

\item \label{thm:imd3}
For all $\signal \in C^\infty_c(\R^d)$ the factors  $\signal_j \coloneqq   \Kx_j \astx  \signal$ satisfy
\begin{align}
\label{eq:imd3a}
    \signal &= \sum_{j  \in \N}
	\Kx_{j}^\plus  \astx f_{j}
	\\ \label{eq:imd3b}
     \Wo \signal_{j} &= \ky_{j} \astt \Wo \signal
	 \,.
\end{align}
Hence any  $\signal \in L^2(\R^d)$ can be recovered from $\Wo \signal$ by first solving \eqref{eq:imd3b}  and then evaluating the series \eqref{eq:imd3a}.
\end{enumerate}
\end{theorem}

\begin{proof}
Follows from  Theorem \ref{thm:cfd}.
\end{proof}

Alternatively, we have the following result  that avoids computing the  canonical dual.

\begin{corollary}[Multiscale reconstruction]\label{cor:cfd}
Let $(\ky_j)_{j \in \N} \in (L^2(\R))^\N$ be a  multiscale decomposition, assume that $\ky_j$ are even functions in $C^{\lceil (d-1)/2 \rceil}(\R) $ with sufficient decay and set $ \Kx_j  \coloneqq  \Ro^\sharp \ky_j$. Then, for all $\signal \in C_c^\infty(\R^d)$, the factors $\signal_j =   \Kx_j \astx  \signal$ satisfy
\begin{align} \label{eq:rep1-m}
	\Phi \astx  \signal &= \sum_{j \in \N}      \Kx_j^\ast \astx
	\signal_j
	\\ \label{eq:rep2-m}
	\Wo  \signal_j &=    \ky_j \astt \Wo \signal  \,,
\end{align}
with $\Phi  \coloneqq     \Fo_d^{-1} \kl{   \sum_{ j  \in \N }     \abs{\Fo_d \Kx_j  }^2  }$.
\end{corollary}

\begin{proof}
Similar to the proof of Theorem~\ref{thm:cfd}.
\end{proof}

Note that the series  in \eqref{eq:imd3a} and \eqref{eq:rep1-m} are both absolutely convergent in  $L^2(\R^d)$; see Remark~\ref{rem:convergence}. From Corollary \ref{cor:cfd} it follows that any function $\signal \in C_c^\infty ( \R^d)$ can be recovered from data $\Wo \signal$ by means of the following consecutive steps:
\begin{itemize}
\item Solve equation \eqref{eq:rep2-m} for $\signal_j$,
\item Evaluate the  series on the right hand side of \eqref{eq:rep1-m},
\item  Solve  the deconvolution problem \eqref{eq:rep1-m}  for $\signal$.
 \end{itemize}
Since $(\ky_\la)_{\la \in \La} $ is a convolution frame, $\Fo_d \Phi$ is bounded away from zero and thus the deconvolution problem \eqref{eq:rep1-m} is stably solvable for the unknown $\signal$. Because the inversion of the wave equation is likewise stable in the full data case, \eqref{eq:rep2-m} can be stably solved for $\signal_j$ as well. However, in the compressed sensing case, this is not the case and we have to incorporate additional prior information to solve these equations.

From Theorem~\ref{thm:imd} we also conclude that the following reasonable strategy can be implemented for CSPAT. Given wave data $\Wo \signal =g$ can be divided into a low-frequency part $ \ky_0 \astt g$ and a high-frequency part $ \sum_{j \geq 1}  \ky_j \astt g$. For the low-frequency part, a standard reconstruction can be employed without any need for regularization. The high-frequency part on the other hand can be reconstructed using CS recovery algorithms. The final fusion is then performed using the dual filters $u_j^\plus$.  An advantage of  approaches based on Theorem~\ref{thm:imd} or Corollary~\ref{cor:cfd} is that they directly lead to sparse elements, while the latter approach requires repeated application of frame analysis and synthesis during iterative CS recovery algorithms.

\begin{remark}[Examples for multiscale decompositions]
Possible multiscale decompositions  can be constructed via a dyadic translation invariant wavelet frame, which names a convolutional frame $(\psi_j)_{j \in \Z}$ where  $\psi_j \coloneqq   2^{j} \psi ( 2^{j} (\edot) )$ for a so-called mother wavelet $\psi \colon \R \to \R$. If we define $\ky_{j} \coloneqq \psi_j$ for $j \geq 1$ and select $\kylow$ such that  $\abs{\Fo_1 \kylow}^2 =  \sum_{j \leq 0} \abs{\Fo_1 \psi_j}^2$  we obtain  a multiscale decomposition $ (\Kx_j) _{j \in \N}$. Alternatively, for the low resolution filter we can   take any other function $\kylow$ such that $\abs{\Fo_1 \kylow}^2 +  \sum_{j \geq  1} \abs{\Fo_1 \psi_j}^2$ is  bounded and away from zero by reasonable constants. Several examples of dyadic translation invariant wavelet frames can be extracted from classical wavelet analysis \cite{daubechies1992ten,mallat2009wavelet}.

Feasible multiscale decompositions can be constructed via a dyadic translation invariant wavelet frame, which denotes a convolutional frame $(\psi_j)_{j \in \Z}$ where $\psi_j \coloneqq 2^{j} \psi ( 2^{j} (\edot) )$ is defined by a so-called mother wavelet $\psi \colon \R \to \R$. If we set $\ky_{j} \coloneqq \psi_j$ for $j \geq 1$ and choose $\kylow$ such that $\abs{\Fo_1 \kylow}^2 = \sum_{j \leq 0} \abs{\Fo_1 \psi_j}^2$ we obtain a multiscale decomposition $ (\Kx_j) _{j \in \N}$. More precisely, \cite[Theorem~1]{chui1993inequalities} implies that if the family $(2^{-j/2}\psi_j(\edot - 2^{-j} k))_{j, k \in \Z}$ is a wavelet frame of $L^2(\R)$, then $(u_j)_{j \in \Z}$ is a multiscale decomposition. Alternatively, for the low-resolution filter, we can take any other function $\kylow$ such that $\abs{\Fo_1 \kylow}^2 + \sum_{j \geq 1} \abs{\Fo_1 \psi_j}^2$ is bounded away from zero by a proper constant. Numerous examples of dyadic translation-invariant wavelet frames can thus be extracted from a classical wavelet analysis \cite{daubechies1992ten,mallat2009wavelet}. Among many others, examples of the generating mother wavelet $\psi = \ky_{\low}$ are the Mexican-Hat wavelet, the Shannon wavelet, the Spline wavelet or the Meyer wavelet. For our numerical simulations, we use the Mexican-Hat wavelet as an arbitrary choice.
\end{remark}

\section{Application to compressed sensing PAT}
\label{sec:cspat}

In this section, we extend the multiscale factorization to the case of compressed sensing data in PAT. We also derive a corresponding sparse recovery scheme.

\subsection{Sampling the wave equation}

In the implementation of any PAT setup, the acoustic data can only be acquired for a finite number of sample points, which we denote by $z_\ell \in S$ for $\ell \in \{1, \dots , n \}$. Note that we do not discretize the temporal variable, since temporal samples can easily be recorded at a rate well above the Nyquist sampling rate.

\begin{definition}[Sampled PAT forward operator]
Let $\Wo$ be the continuous PAT forward operator defined in \eqref{eq:Wo}. For sampling points  $ \zz_1, \dots, \zz_n \in \surface$, we set
\begin{align}
&\So_n \colon C^\infty (\surface \times (0,\infty)) \to  (C^\infty  (0,\infty) )^n \colon g   \mapsto  \left(  g (\zz_\ell, \, \cdot\, ) \right)_{\ell= 1, \dots, n }
\\ \label{eq:Wn}
 &\Wo_n \colon C^\infty_c (\R^d) \rightarrow   (C^\infty  (0,\infty) )^n
    \colon \signal   \mapsto    \So_n   \Wo \signal  = \left(  \Wo \signal (\zz_\ell, \, \cdot\, ) \right)_{\ell = 1, \dots , n   } \,.
\end{align}
We call  $\So_n$ the regular sampling scheme and  $\Wo_n = \So_n   \Wo$ the (regularly) sampled  PAT forward operator corresponding to the $n$-tuple $( \zz_\ell )_{\ell = 1, \dots , n }$ of spatial sampling points.
 \end{definition}

The fundamental question of classical sampling theory in the context of PAT is to find ``simple'' and ``reasonable'' sets $\X$ to which the initial pressure belongs, and corresponding conditions on the sampling points under which the sampled data $\Wo_n \signal$ uniquely and stably determine the initial pressure distribution $\signal \in \X$. For equidistant detectors located on the boundary of a circular disk $D \subseteq \R^2$, explicit sampling conditions for PAT have been derived in \cite{haltmeier2016sampling}.  Roughly speaking, these results state that any function $\signal \in C^\infty_0( D )$ whose Fourier transform $\hat \signal (\xi) $ is sufficiently small for $\norm{\xi} \geq b_0$, where $b_0$ is the essential bandwidth, can be stably recovered from sampled PAT data $\Wo_n \signal$, provided that the sampling condition $n \geq 2 R_0 b_0$ is satisfied (for a precise statement, see \cite{haltmeier2016sampling}).  Sampling theory for other tomographic inverse problems is treated, for example, in \cite{stefanov2020semiclassical,katsevich2017local,Far06,desbat1993efficient,natterer95sampling,nguyen2020sampling}.

{
We have the impression that sampling theory in PAT has not yet received as much attention as it would deserve.  It is both of practical relevance and of mathematical interest. However, only a few special cases exist in which it has actually been solved. Since the correct sampling of the forward operator is, in a sense, the starting point of our compressive approach, we would like to discuss this topic in some more detail here.
To our knowledge, \cite{haltmeier2016sampling} is the earliest reference explicitly dealing with the sampling of the PAT operator $\Wo$. The analysis presented there applies to the case where the measurement surface is a circle in 2D.  The derivation closely follows the presentation of the sampling theory for the 2D Radon transform in \cite[Section III.3]{natterer1986computerized}. The main step in this approach is to estimate the essential support of the 2D Fourier transform of $\Wo \signal$, see \cite[Theorem~6]{haltmeier2016sampling}. In particular, 2D sampling schemes based on classical Shannon sampling theory are derived. Using non-uniform sampling theory \cite{landau1967necessary}, one could derive non-uniform sampling schemes from the support estimate. Such studies in the context of PAT would be an interesting line of research. Furthermore, the extension of the approach to the case of general detection surfaces as well as to higher dimensions seems to be interesting. We mention at this point a recent work \cite{mathison2020sampling} along this line, based on microlocal analysis.}

\subsection{Compressive sampling}

In order to reduce the number of detectors while maintaining spatial resolution, CSPAT has been investigated in several works \cite{arridge2016accelerated,guo2010compressed,haltmeier2016compressed,provost2009application,sandbichler2015novel}.  The basic idea is to use general linear measurements of the form
\begin{align}\label{eq:CS-measurements}
\data_j = \inner{a_j} {\Wo_n (\signal)} = \sum_{i=1}^n a_{j, i} (\Wo_n (\signal))_i
 \quad \text{ for } j \in \left\{1,\dots, m \right\} \,.
\end{align}
Here $a_j$ are measurement vectors with entries $a_{j,i}$, and $\Ao_{m,n} \coloneqq (a_{j,i})_{j,i} \in \R^{m \times n}$ is the measurement matrix. The term compressed sensing refers to the fact that the number of measurements $m$ is to be chosen much smaller than the number of initial sampling points $n$. Therefore, $\data = \Ao_{m,n} \Wo_n \signal$ is a highly underdetermined linear system of equations and can only be solved with additional information on the unknown to be recovered.

For a systematic treatment, we introduce the following notation.

\begin{definition}[Generalized PAT sampling] \label{def:fwd-cs}
For sampling points $\zz_1, \dots, \zz_n  \in \surface$, measurement matrix $\Ao_{m,n} \in \R^{m \times n}$ and subspace  $\X \subseteq C_c^\infty(\R^d)$ we call
\begin{enumerate}
\item $\Ao_{m,n} \So_n$ a generalized sampling scheme;
\item $\Ao_{m,n} \Wo_n = \Ao_{m,n} \So_n \Wo $ CSPAT forward operator;
\item $\Ao_{m,n}\So_n$ a complete sampling scheme for $\X$, if the restriction $\Ao_{m,n} \So_n \Wo|_{\X}$ is injective.
\end{enumerate}
\end{definition}

The results of \cite{haltmeier2016sampling} basically show that $\So_n$ provides a complete sampling scheme on the space $V = V_{R_0, b_0}$ of all functions supported in a disk of radius $R_0$ and having a essential  bandwidth $b_0$ given $n \geq 2 R_0 b_0$ equally distributed sampling points. This implies that for any invertible matrix $\Ao_{m,n}$ the composition $\Ao_{m,n} \So_n$ is also a complete sampling scheme on $V_{R_0,b_0}$. We are not aware of any  available results if the generalized sampling scheme cannot be written in the form $\Bo \So_m$, where $\So_m$ is a regular sampling scheme and  $\Bo  \in \R^{m \times m}$  is invertible. The question for which spaces general sampling matrices provide complete sampling schemes seems to be an interesting line of open research. Anyway, in this paper we study the case where $\Ao_{m,n}  \Wo_n$ is not injective on the linear subspace $\X$ and develop a nonlinear reconstruction approach based on $\ell^1$-minimization.

\subsection{Multiscale reconstruction for CSPAT}
\label{ssec:ms}

Let    $\Wo_n = \So_n   \Wo$ be  a regularly sampled  PAT forward operator with sampling points $\zz_1, \dots, \zz_n  \in \surface$. We suppose that the regular sampling scheme $\So_n$ is complete for a subspace  $\X_n \subseteq  C_c^\infty (\R^d)$.
Moreover, let $\Ao_{m,n} \in \R^{m \times n}$ be a measurement matrix with $ m < n$ such $\Ao_{m,n} \So_n$ is not complete on $\X_n$. This means that $\Ao_{m,n} \Wo_n$ is not injective on $\X_n$ and  therefore cannot be uniquely inverted. Our aim  is to nevertheless to recover $\signal \in \X_n$ from data $\data = \Ao_{m,n} \Wo_n \signal $ by using  suitable prior information.

In the following, we describe how a multiscale factorization for the wave equation can be used to recover the initial pressure from CSPAT data. The main ingredient of the approach is that the factorizations of Theorem~\ref{thm:imd} and Corollary~\ref{cor:cfd} for the full wave equation generalize to the compressed sensing setup. We formulate here only one such extension based on Corollary~\ref{cor:cfd} because we will use this version for the actual numerical implementation.

\begin{proposition}[Multiscale CSPAT decomposition]\label{cor:csrec}
Let $( \conv{\Kx}, \conv{\ky})$ be a multiscale factorization of $\Wo$ and set $\Phi  \coloneqq     \Fo_d^{-1} \bigl(   \sum_{ j  \in \N }     \abs{\Fo_d \Kx_j  }^2  \bigr)$. For all $\signal \in C_c^\infty (\R^d)$ the factors $\signal_j = \Kx_j \astx \signal$ satisfy
\begin{align} \label{eq:csrec1}
	\Phi \astx  \signal &= \sum_{j \in \N}   (\Ro^\sharp \ky_j)  \ast \signal_j
    \\
    \label{eq:csrec2}
	(\Ao_{m,n} \Wo_n)  \signal_j &=    \ky_j \astt ( \Ao_{m,n} \Wo_n  \signal ) \,.
\end{align}
\end{proposition}

\begin{proof}
This follows from Corollary \ref{cor:cfd} by using that  $\Ao_{m,n} \So_n$ acts in the spatial variable and therefore commutes  with the temporal convolution.
\end{proof}

\subsection{Proposed multiscale reconstruction}
\label{ssec:MSrecon}

Consider the situation as stated in Subsection~\ref{ssec:ms}. Based on the factorization of Corollary~\ref{cor:csrec}, we propose the following scheme for reconstructing an initial pressure from  CSPAT data :

\begin{alg}[Reconstruction of the initial pressure $\signal$ from CSPAT data $\Ao_{m,n} \Wo_n \signal$]\label{alg:cs}\mbox{}
\begin{enumerate}[label=(S\arabic*), leftmargin=4em]
\item \label{rem:s1} Solve equation \eqref{eq:csrec2} for $\signal_j$ as described below.
\item \label{rem:s2} Evaluate the  series on the right hand side of \eqref{eq:csrec1}.
\item \label{rem:s3} Solve  the deconvolution problem \eqref{eq:csrec1} for $\signal$.
 \end{enumerate}
\end{alg}

Since $(\ky_j)_{j \in \N} $ is a convolutional frame, $\Fo_d \Phi$ is bounded away from zero and therefore the deconvolution problem \eqref{eq:csrec1} in step~\ref{rem:s3} is stably solvable. However, since $\Ao_{m,n} \Wo_n$ is not injective, solving \eqref{eq:csrec2} in step~\ref{rem:s1} requires the use of available prior information.  The proposed  solution procedure is described in the following Remark~\ref{rem:ss}.

\begin{remark}[Solution of step~\ref{rem:s1}] \label{rem:ss}
Reasonable prior information differs for the low-frequency factor $f_0$ and the high-frequency factors $\signal_j$ for $j \geq 1$. Therefore, we recover the low-frequency factor and the high-frequency factors in different ways.
\begin{itemize}
\item \textsc{Low-frequency factor:}  Assuming that the low-frequency filter $ \ky_0$ has essential bandwidth in  $b_0$, then  low-frequency factor $ f_0 = \Kx_0 \astx  \signal$ also has essential bandwidth $b_0$. This means that the Fourier transform $\Fo_d f_0(\xi)$ is sufficiently small for $\norm{\xi} \geq b_0$. Therefore, classical sampling theory in the context of PAT \cite{haltmeier2016sampling} suggests that $m$ regular samples are sufficient to recover $ f_0 = \Kx_0 \astx  \signal$ as solution of
\begin{equation} \label{eq:ell2}
	\min_{h \in \X_n}
	\norm{h}_2  \quad \text{ such that }
	\Ao_{m,n} \Wo_n h =  \data_0  \,.
\end{equation}
We will use \eqref{eq:ell2}  when $\Ao_{m,n} \in \R^{m \times n}$ is a  either a  subsampling matrix or random sensing matrix.

\item
\textsc{High-frequency factors:} Under the assumption that the Fourier transform $\Fo_1 \ky(\omega)$ is negligible in a suitable sense if $\abs{\omega}$ is outside the interval $[b_0 , 2 b_0]$, then the Fourier transforms  $\Fo_d \Kx(\xi)$ of high-frequency factors are negligible outside the ring $D_j \coloneqq \set{\xi \in \R^d \mid 2^{j-1} b_0 \leq \norm{\xi} \leq 2^j b_0}$. However, if we perform compressive sampling, then $\Ao_{m,n} \Wo_n$ is not injective on the space of all functions whose Fourier transform is essentially supported in $D_j$, and therefore \eqref{eq:csrec1} cannot be uniquely inverted without additional prior information. As we can observe from Figure~\ref{fig:sparse}, the high-frequency factors $f_j = \Kx_j \astx \signal$ are sparse in the spatial domain. Therefore, in this work, we propose to use $\ell^1$-minimizing solutions
\begin{equation} \label{eq:ell1}
	\min_{h \in \X_n}
	\norm{h}_1 \quad \text{ such that }
	\Ao_{m,n} \Wo_n h = \data_j \,.
\end{equation}
Here $\norm{h}_1 \coloneqq \sum_{i \in \Z^2 } \abs{ h( x_i ) }$ the $\ell^1$-norm of $h$ with discrete samples $x_i \in \R^2$.
\end{itemize}
\end{remark}

We can infer uniqueness of \eqref{eq:ell1} based on recovery conditions of compressed sensing. Such results also imply stable recovery in the case of approximately sparse signals and noisy data, with \eqref{eq:ell1} replaced by relaxed versions. The standard condition that guarantees stable recovery of sparse signals is the restricted isometry property (RIP). The RIP guarantees uniform recovery of all sufficiently sparse signals \cite{candes2008restricted}, which is unexpected for the specific sampling matrices considered here, where compressive measurements are performed only in the spatial dimension. In such a situation, we can resort to the results of \cite{grasmair2011necessary}, dealing with the reconstruction of individual elements. A detailed  error analysis for CSPAT with Algorithm~\ref{alg:cs} using \eqref{eq:ell2} for the low resolution factor and \eqref{eq:ell1} for the high resolution factors is an interesting line of future research and beyond the scope of this paper.

\section{Numerical experiments}
\label{sec:num}

In this section, we present details on the implementation of the sparse reconstruction scheme from CSPAT measurements presented in Section~\ref{sec:cspat}. In our numerical experiments, we consider the two-dimensional case when the initial pressure is supported in the circular disk in $\R^2$ with radius $0.9$ and the measurements are performed in the unit sphere $\sph^1$. This situation occurs in PAT with integrating line detectors~\cite{burgholzer2007temporal,finch2007spherical,paltauf2009characterization}.

\begin{figure}[htb!]
	\centering
	\includegraphics[width=0.48\textwidth]{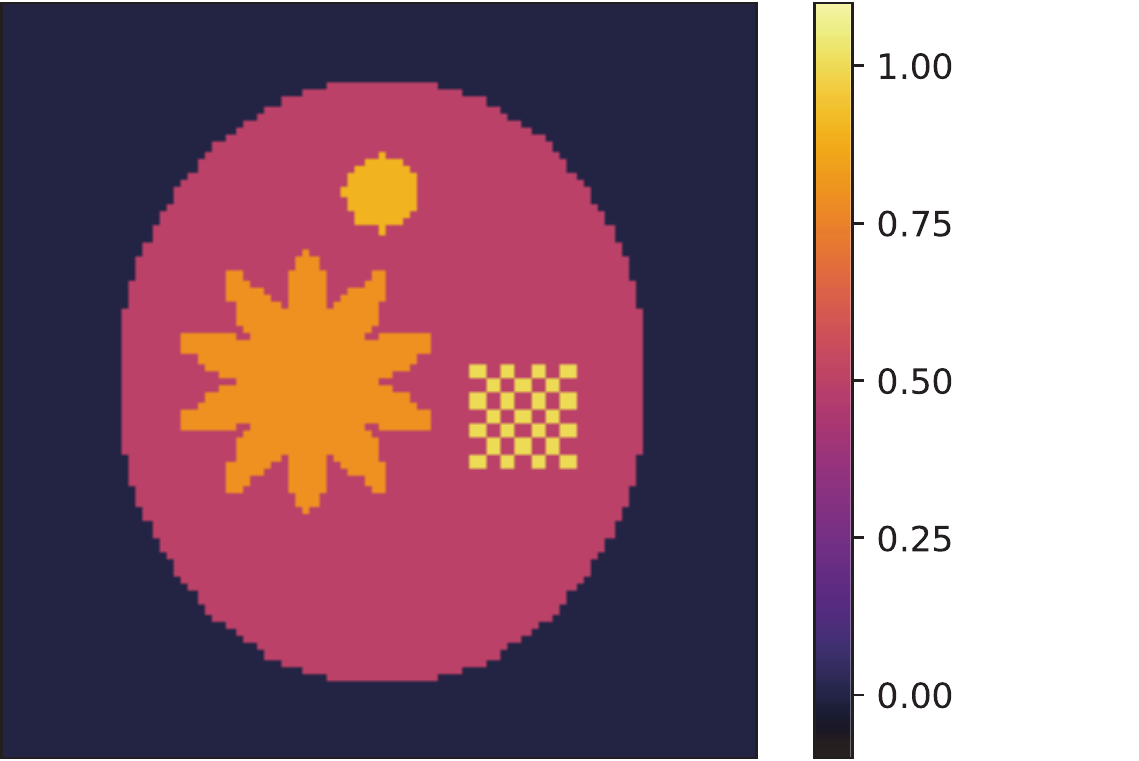}\quad
	\includegraphics[width=0.48\textwidth]{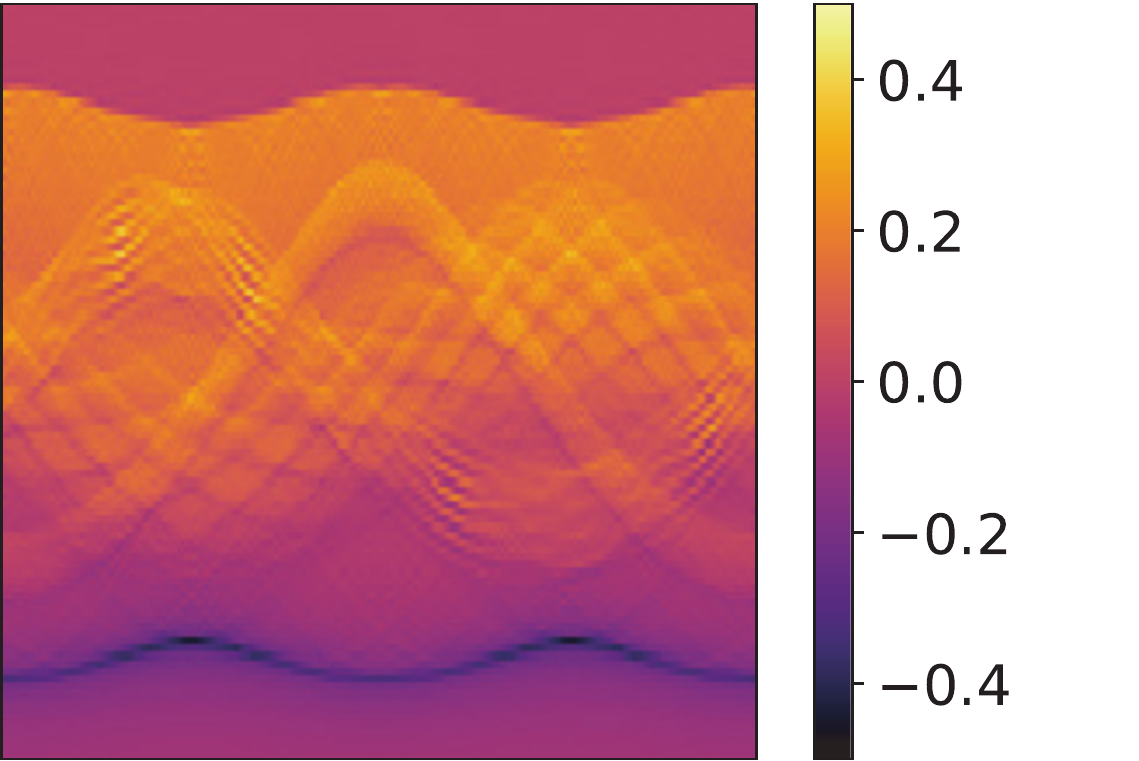}\\[1em]
	\includegraphics[width=0.48\textwidth]{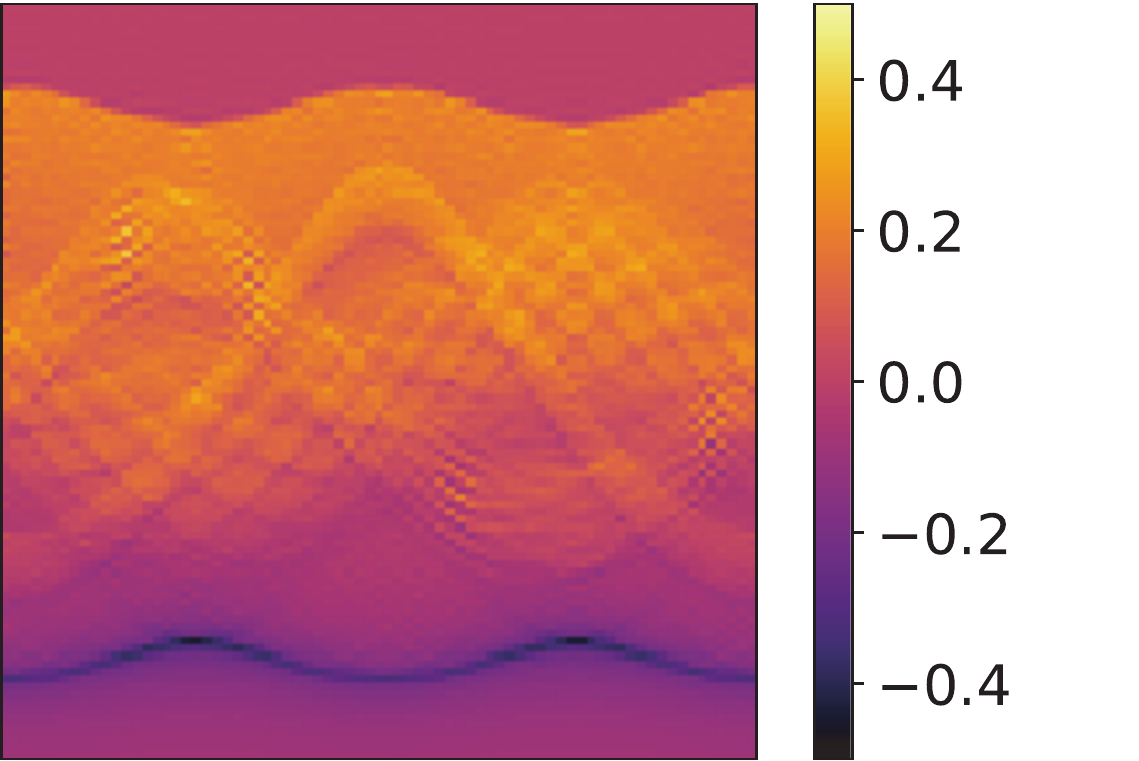}\quad
	\includegraphics[width=0.48\textwidth]{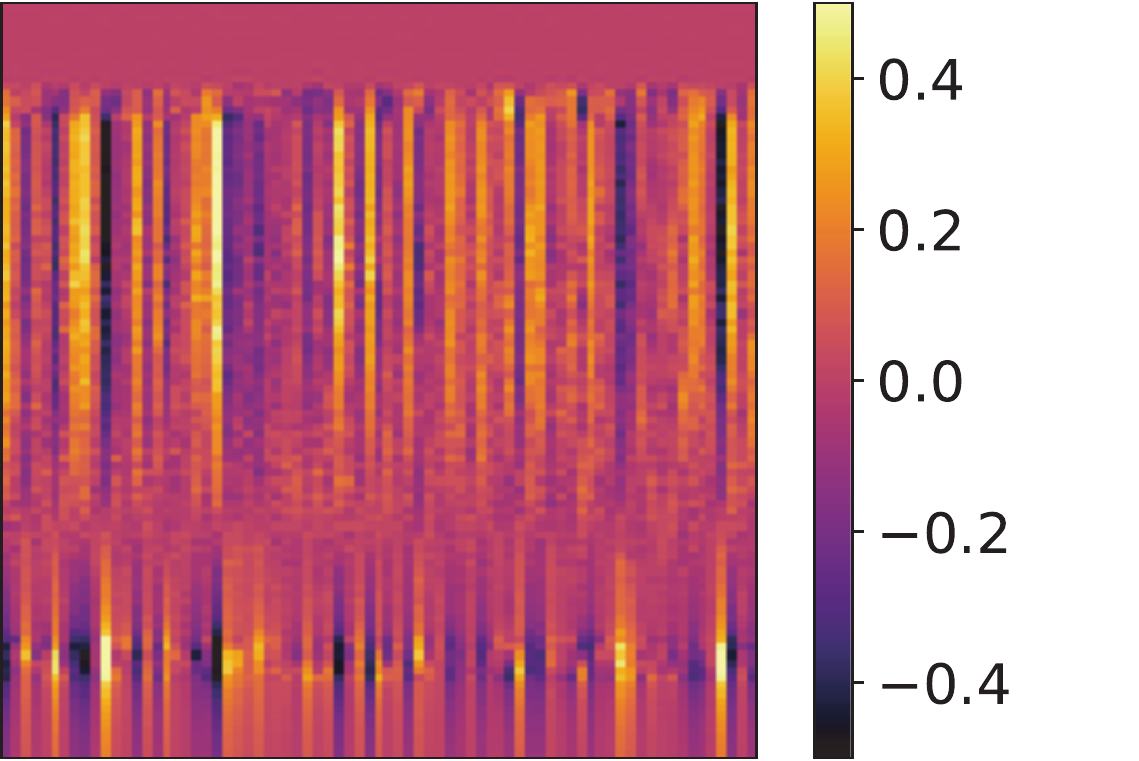}
	\caption{Left: Initial pressure on a square grid of side length $2$. The detectors are equidistantly distributed on the black circle. Right: Full data.}
	\label{fig:data}
\end{figure}

\subsection{Numerical implementation} \label{sec:numdatasim}

For all presented numerical implementations, we represent $\signal$ by its discrete values $ (\signal(x_i))_{i}$ on a Cartesian grid of side length $2$ at nodes $x_i = i \, 2/\Nx$ for $i \ in \{-\Nx/2, \dots,  \Nx/2 -1 \}^2$ with $\Nx = 100$.  The space $\X_n$ is taken as the space of all piecewise bilinear functions with nodes at $x_i$ whose values vanish outside the disk of radius $0.9$.  We implement $\Wo_n$ and $\Ao_{m,n}$ with $n=300$ and $m=75$ as described below. We assume that $\signal$ is sampled at the Nyquist rate such that  the maximal Bandwidth is given by $\Omega \coloneqq  \Nx (\pi /2)$.  Note that the fully sampled PAT forward operator $\Wo_n$ satisfies the classical sampling conditions.    The measurement matrix $\Ao_{m,n}$ corresponds to a subsampling factor of 4.

\begin{itemize}
\item \textsc{Sampled PAT forward operator and adjoint:}
The discretization of $\Wo_n$ is based on the Fourier representation $ \Fo_d \, p ( \xxi , t) = \cos ( \norm{\xxi} t ) \Fo_d \signal (\xxi)$ for the solution of the wave equation \eqref{eq:wavee1}-\eqref{eq:wavee3}.  For the numerical computations, we replace the Fourier transform by the discrete Fourier transform on the square grid of side length $4$ with spatial nodes $\xx_{ i } = i \, 2/\Nx$ for $i \in \{- \Nx , \dots , \Nx -1 \}^2$ and frequency nodes $\xxi_k = k \, \Omega / \Nx$ for $ k \in \{-\Nx, \dots , \Nx - 1 \}^2$. Here, the bandwidth $\Omega$ and the spatial sampling step size $2/\Nx$ satisfy the Nyquist condition $2/\Nx = \pi /\Omega$ and the larger numerical domain $[-2, 2] \times [-2, 2] $ is chosen to avoid boundary effects. We then define the discrete fully sampled PAT forward operator $\Wo_n$ by nearest neighbor interpolation at the detector locations.  The adjoint $\Wo_n^\herm$ is numerically computed using the backprojection algorithm described in \cite{burgholzer2007temporal}.

\item  \textsc{Multiscale filters:}
The high-frequency filters $\ky_j$ for $ \geq 1$ are chosen as Mexican Hat wavelets
\begin{equation*}
\ky_j(t)
\coloneqq 8 \cdot 2^{j} (1 - (2^j 8 t)^2 )
 \exp \left( -\frac{ ( 8 \cdot 2^j t)^2}{2} \right)
\end{equation*}
and the corresponding low-frequency temporal filter $ \kylow$ is taken as the Gaussian kernel $\kylow(t) \coloneqq 8 \exp ( - (8t)^2/2 )$.   The width of the filter $\kylow$ has been chosen such that $\kylow \astx \signal$ can be recovered from $n = 75$ samples according to classical sampling theory. For the high-frequency components $\ky_j \astx \signal$ with $j \geq 1$ this is not the case, and therefore we use sparsity as outlined in Section~\ref{sec:cspat}. The spatial filters $\Kx_j = \Ro^\sharp \ky_j$ are computed analytically by evaluating \eqref{eq:reciD} with $d=2$ and $\ky = \ky_j$. For $j \geq 2$, the essential support $\ky_j$ lies outside the considered frequency regime $[-\Omega, \Omega]$ and therefore we restrict to the three filters $\kylow, \ky_1, \ky_2$ for the numerical simulations. All temporal and spatial convolutions are replaced by discrete convolutions computed via the discrete Fourier transform.

\item \textsc{Measurement matrix:}
For the  measurement matrix $\Ao_{m,n} \in   \R^{m \times n}$ we consider two choices.  First, we  take $\Ao_{m,n}$ as uniform subsampling matrix which has entries  $a_{j,i} = 1$ if  $j= 4(i-1)+1$ and  $a_{j,i} = 0$ otherwise. Second we take $\Ao_{m,n}$ as Gaussian random matrix where each entry $a_{j,i}$ is the realization of an independent Gaussian random variables with zero mean.
\end{itemize}

\begin{figure}[htb!]
\centering
\includegraphics[width=0.48\textwidth]{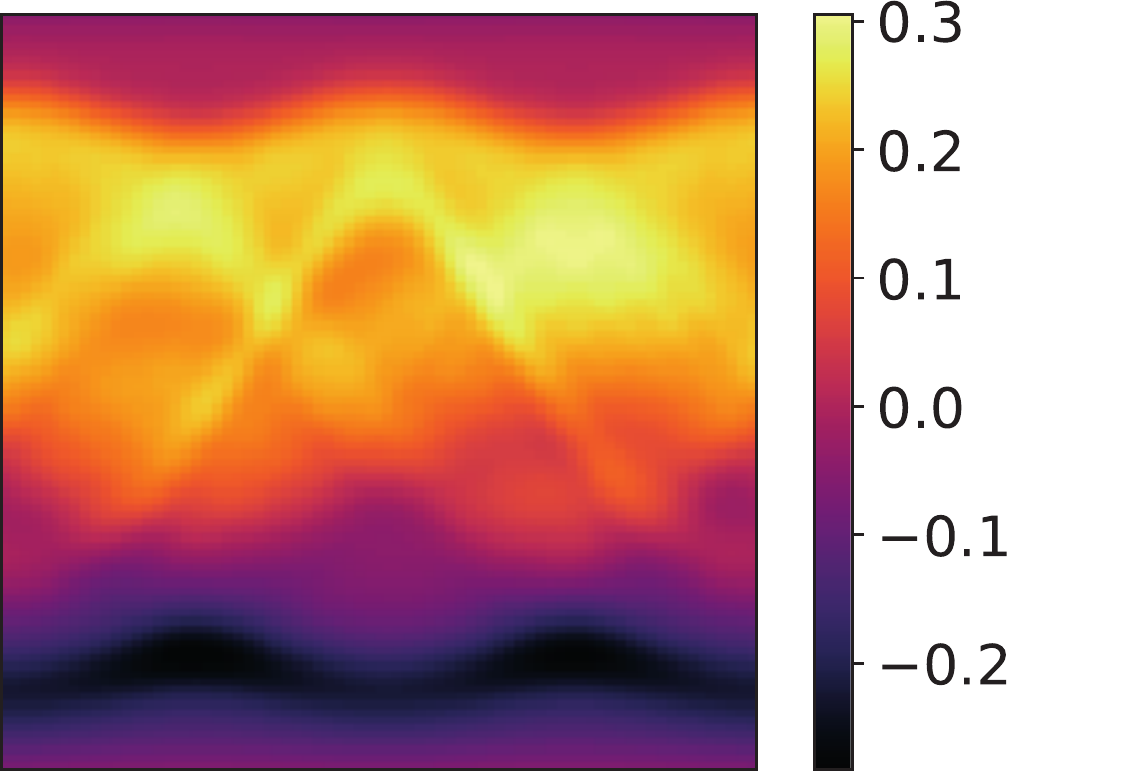}
\includegraphics[width=0.48\textwidth]{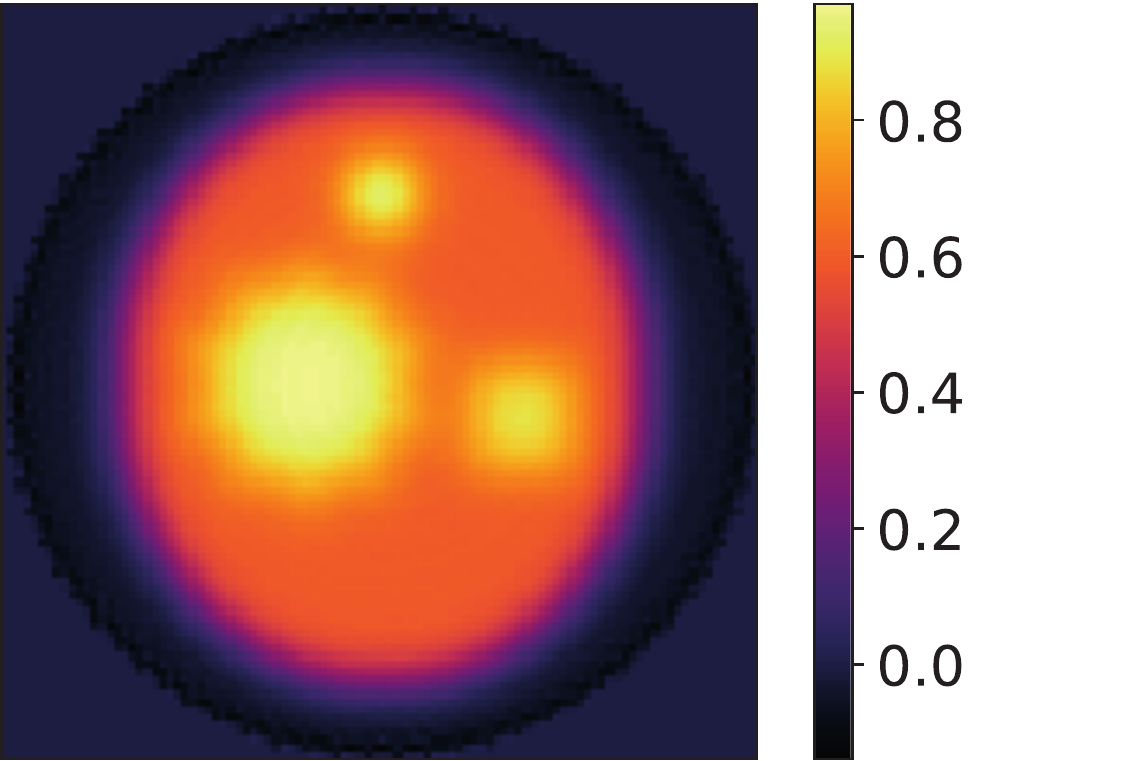}\\[1em]
\includegraphics[width=0.48\textwidth]{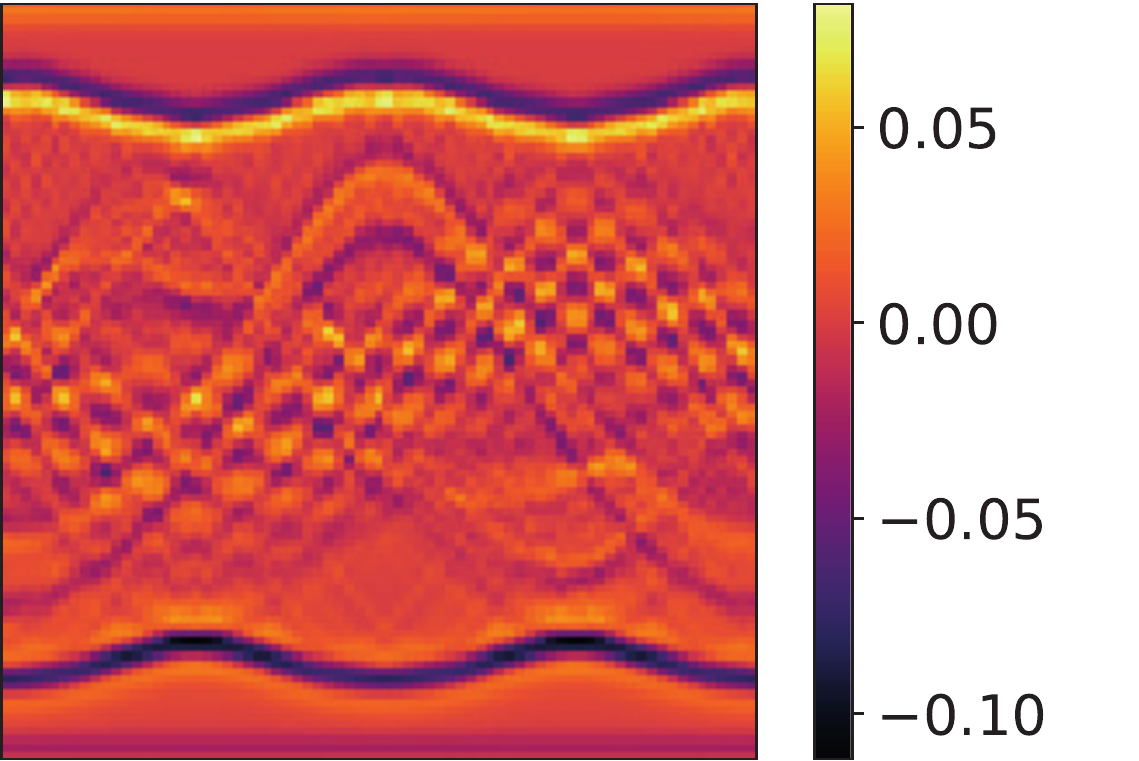}
\includegraphics[width=0.48\textwidth]{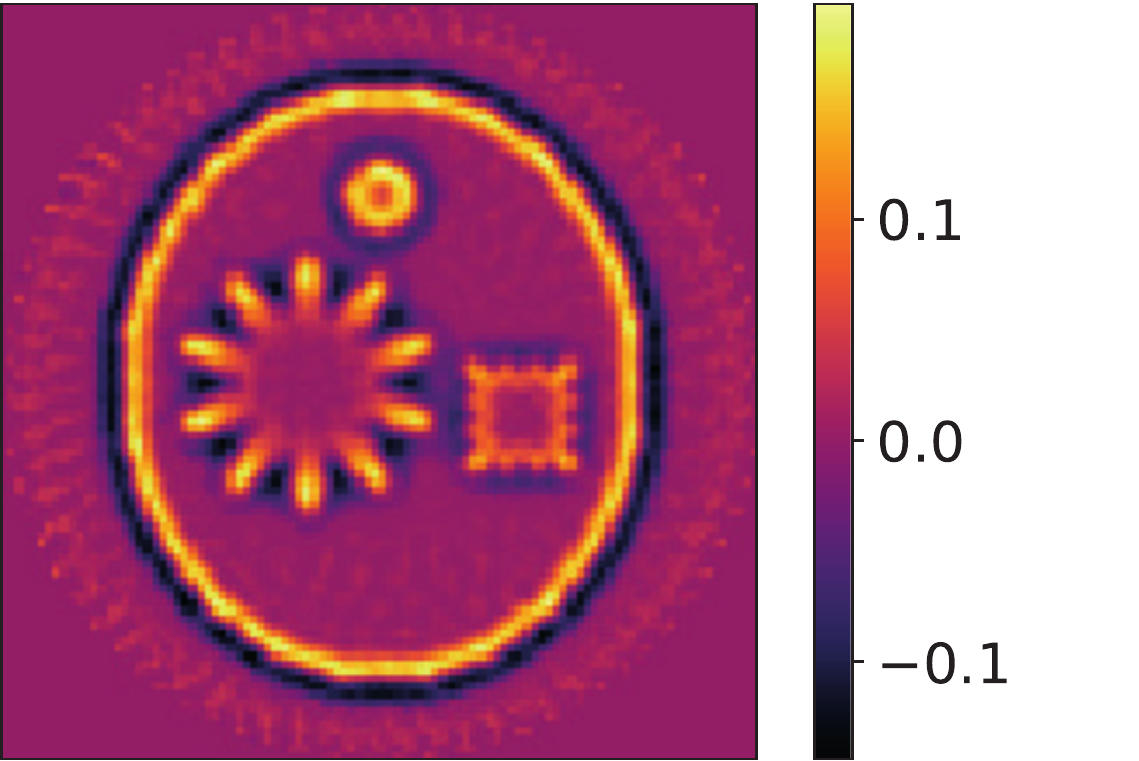}\\[1em]
\includegraphics[width=0.48\textwidth]{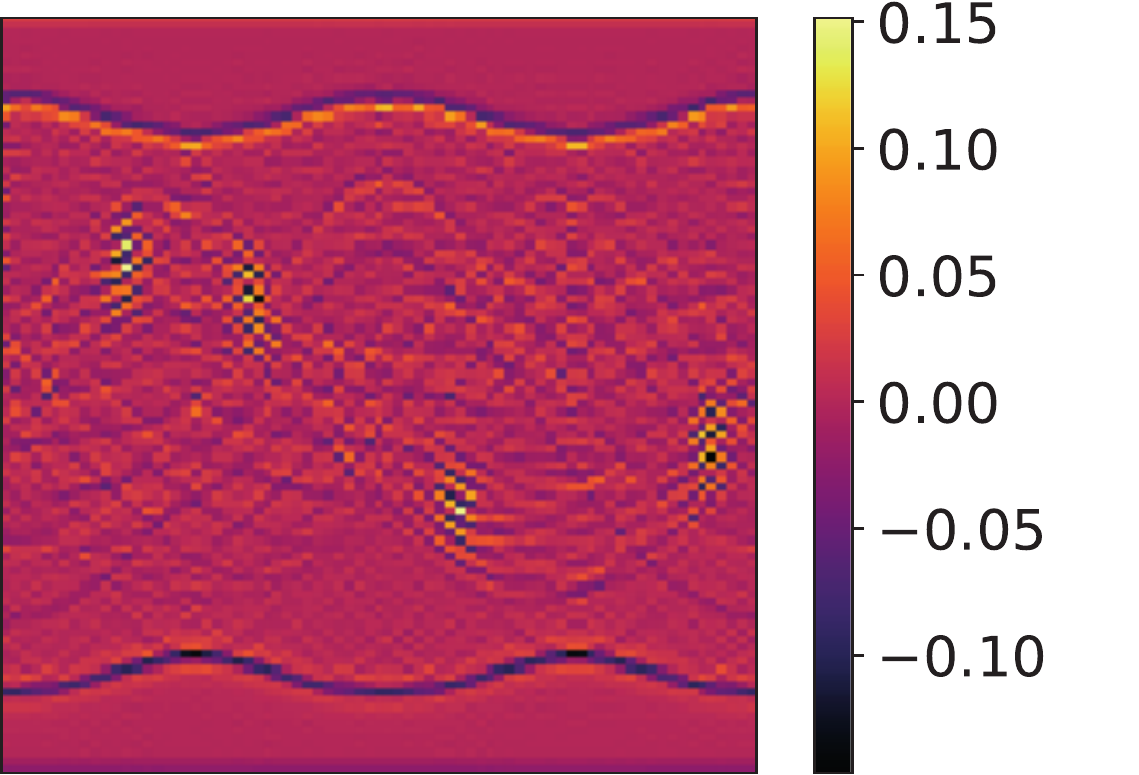}
\includegraphics[width=0.48\textwidth]{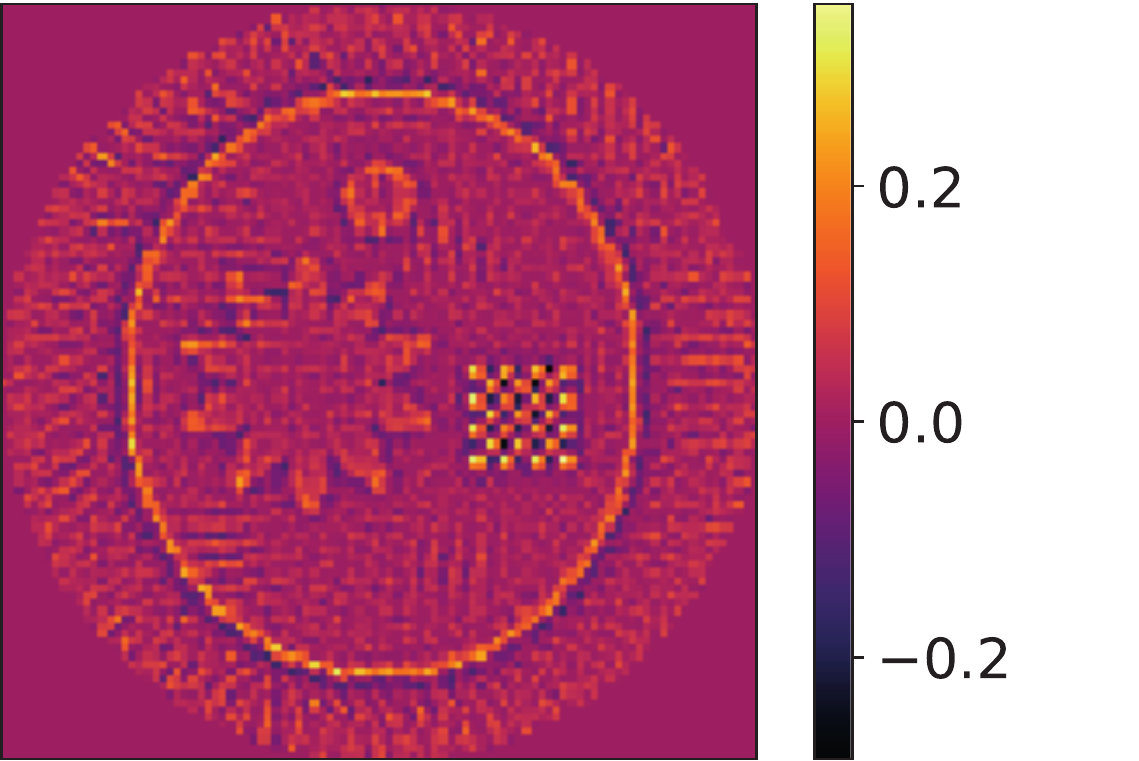}
\caption{Left: The convolved data $ \data_j = \ky_j \astt (\Ao_{m,n} \Wo_n \signal) $ for $j=0,1,2$ where $\Ao_{m,n}$ is the  subsampling matrix with subsampling factor 4. Right.  Corresponding reconstructions  of convolved initial  pressure $\Kx_j \astx \signal$ using the Landweber method (for $j=0$) and iterative soft thresholding (for $j=1,2$).}\label{fig:rec-sparse}
\end{figure}

\begin{figure}[htb!]
\centering
\includegraphics[width=0.48\textwidth]{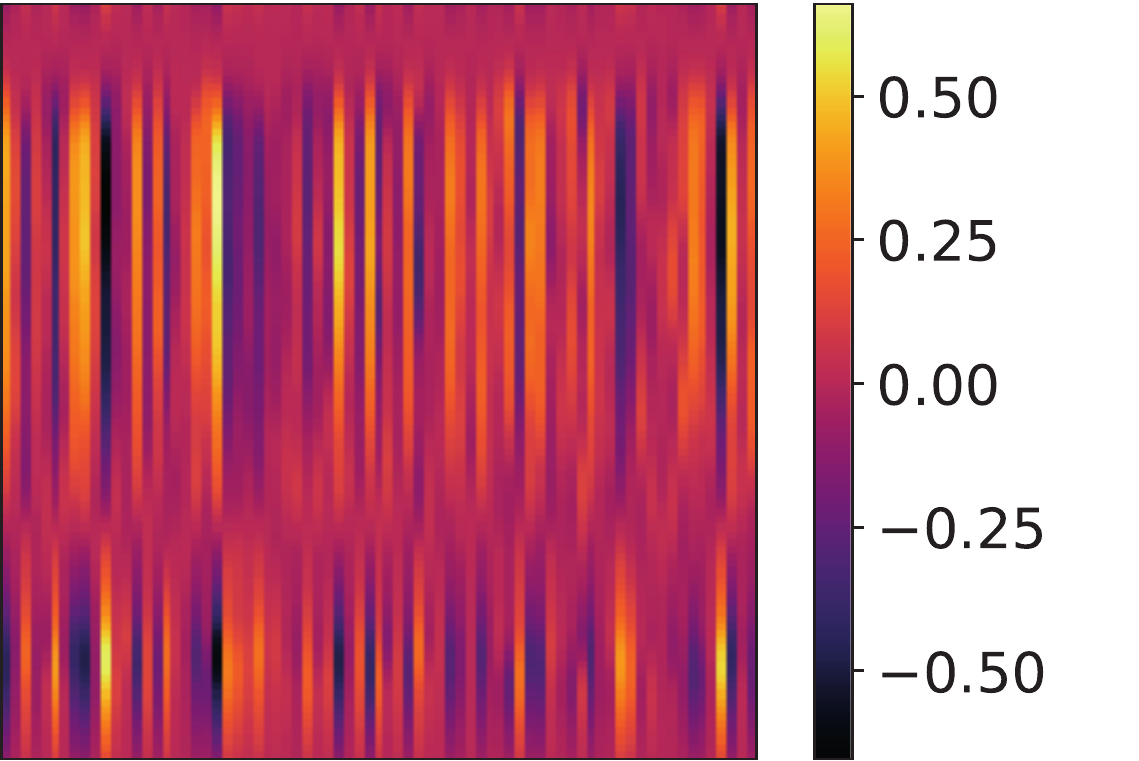}
\includegraphics[width=0.48\textwidth]{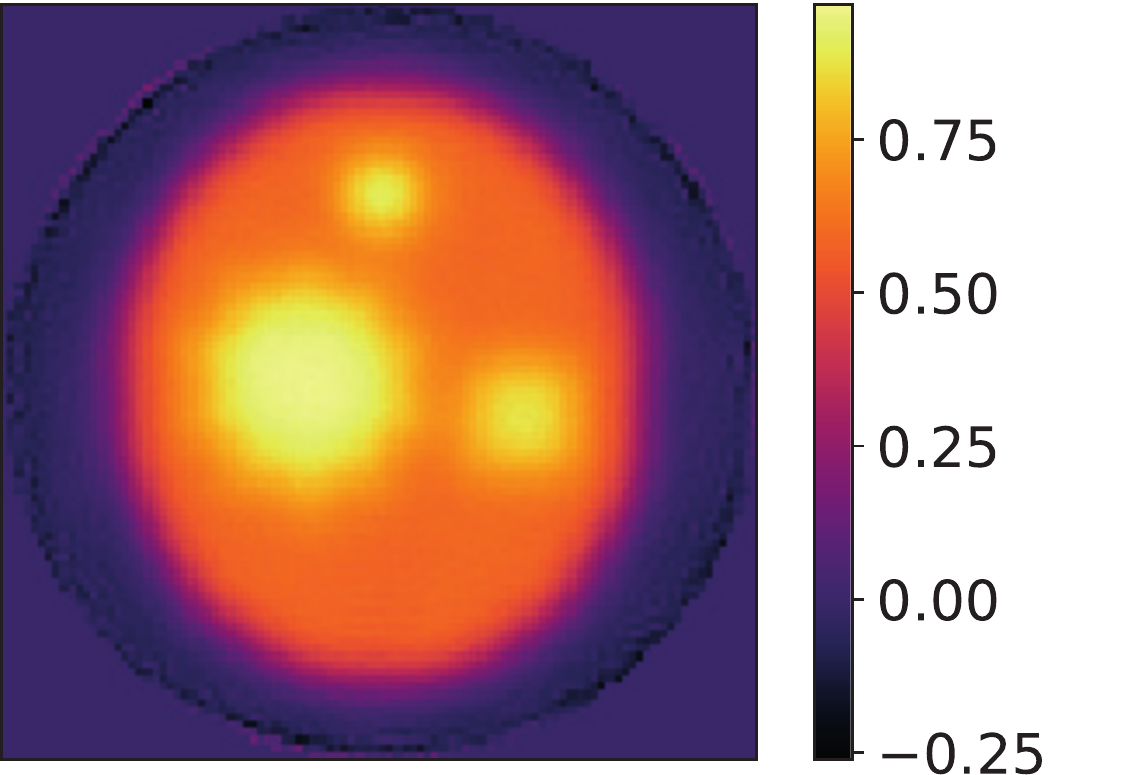}\\[1em]
\includegraphics[width=0.48\textwidth]{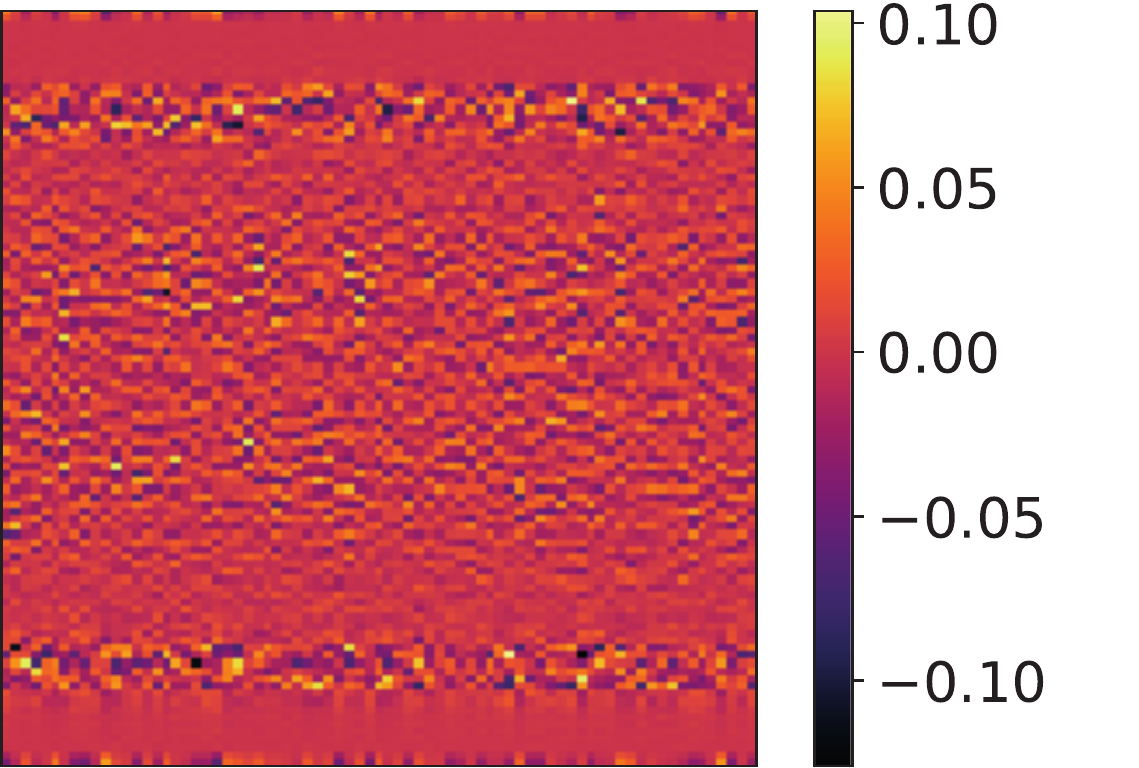}
\includegraphics[width=0.48\textwidth]{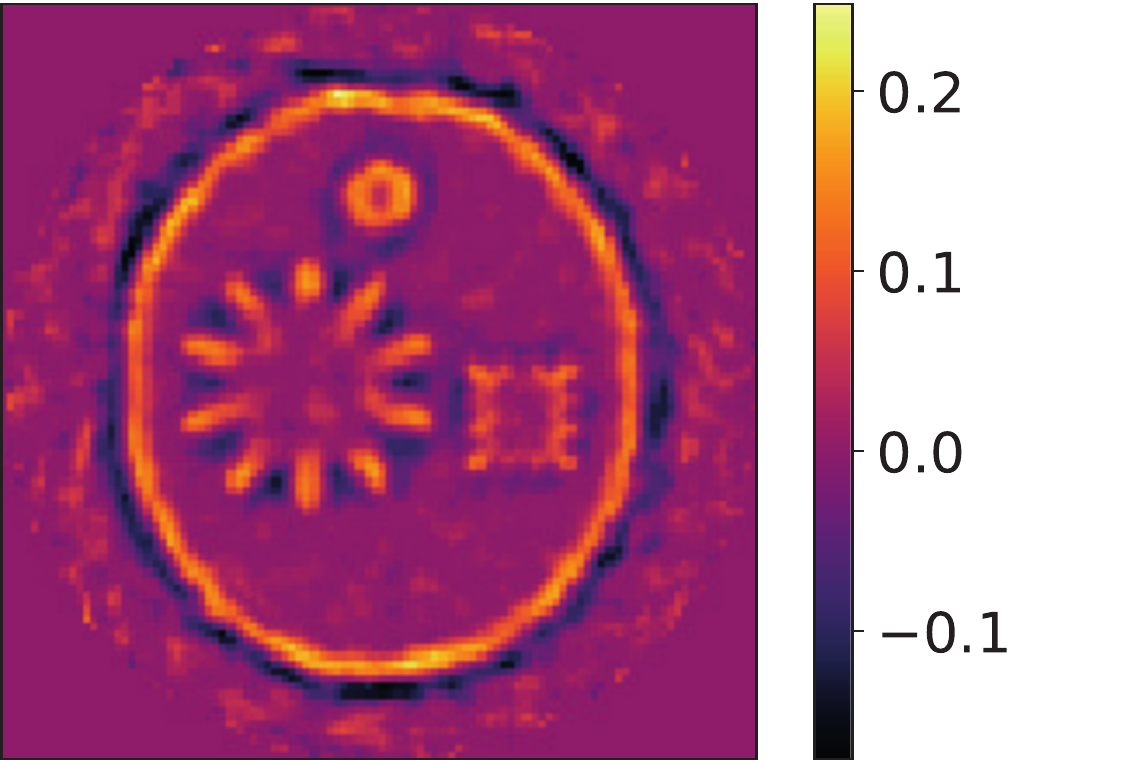}\\[1em]
\includegraphics[width=0.48\textwidth]{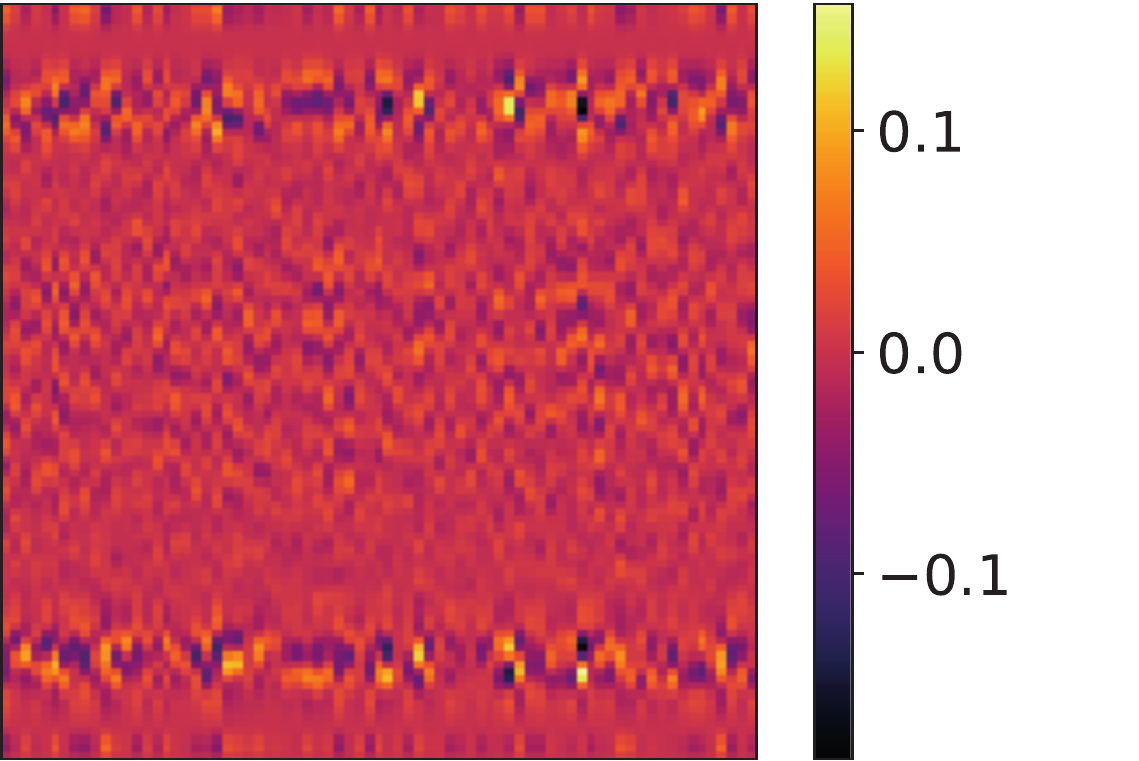}
\includegraphics[width=0.48\textwidth]{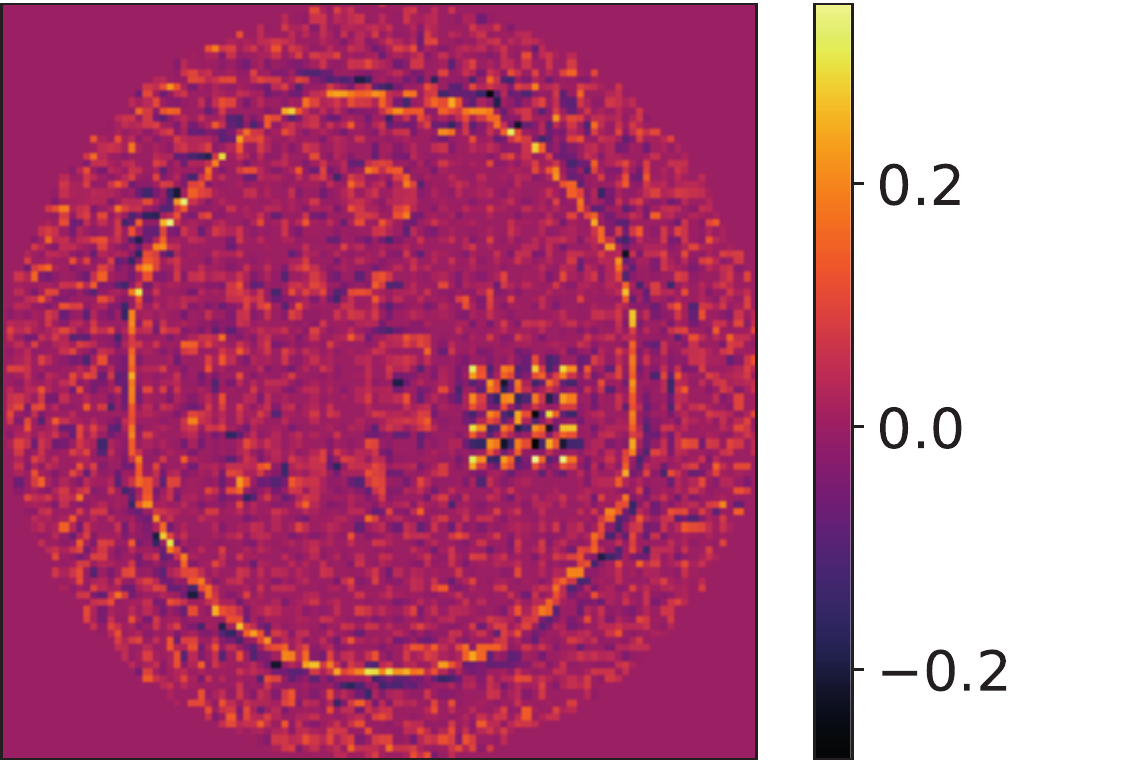}
\caption{Left: The convolved data $ \data_j = \ky_j \astt (\Ao_{m,n} \Wo_n \signal) $ for $j=0,1,2$ where $\Ao_{m,n}$ is a Gaussian random matrix with subsampling factor 4. Right.  Corresponding reconstructions  of convolved initial  pressure $ \Kx_j \astx \signal  $ using the Landweber method (for $j=0$) and iterative soft thresholding (for $j=1,2$).}\label{fig:rec-gauss}
\end{figure}

The initial pressure used for the numerical simulations,  the corresponding fully  sampled data as well as the subsampled  and the Gaussian measurement data are shown in Figure~\ref{fig:data}. The filtered data for the subsampling scheme are shown in the left column of Figure~\ref{fig:rec-sparse} and the filtered data for the Gaussian measurements are shown in the left column of Figure~\ref{fig:rec-gauss}.

\subsection{Reconstruction results}\label{sec:numrec}

Following the strategy proposed in Section~\ref{sec:cspat} (see Algorithm \ref{alg:cs} and Remark~\ref{rem:ss}),  we  recover the initial phantom via the following three steps:

\begin{itemize}
\item First recover the factors $\Kx_j \astx \signal$ from data $\data_j = \ky_j \astx  (\Ao_{m,n} \Wo_n \signal)$. For that purpose, we use the Landweber iteration for recovering the low-frequency factor $\Kx_0 \astx \signal$ and the iterative soft thresholding algorithm
\begin{equation*}
	\forall  k \in \N  \colon \quad
	h^{k+1} = \soft_{ s \lambda} \kl{ h^k  -   s  \,   \Wo_n^\herm \Ao_{m,n}^\herm\left(   \Ao_{m,n} \Wo_n  h^k - \data_j \right) }
\end{equation*}
for recovering the sparse  high-frequency factors $\Kx_1 \astx \signal$, $\Kx_2 \astx \signal$. Here $\soft_{s \lambda}  \signal =  \sign( \signal )  \,  \max\{  \abs{\signal}  - s \lambda , 0 \}$ is the soft thresholding operation, $s$ is the step size and $\lambda$ is the regularization parameter.

\item Second we evaluate  $f_{\rm conv} \coloneqq \sum_{j=0}^2  \Kx_j^\ast \astx \signal_j$.

\item
As final reconstruction step we recover an approximation to  $\signal$ by  deconvolution  $f_{\rm conv}$ with kernel  $\Phi = \Fo_d^{-1} \sum_{j= 0 }^2   \abs{ \Fo_d \Kx_j  }^2$. In this work we again use the  iterative soft thresholding algorithm for implementing the deconvolution.
\end{itemize}

\begin{figure}[htb!]
\centering
\includegraphics[width=0.48\textwidth]{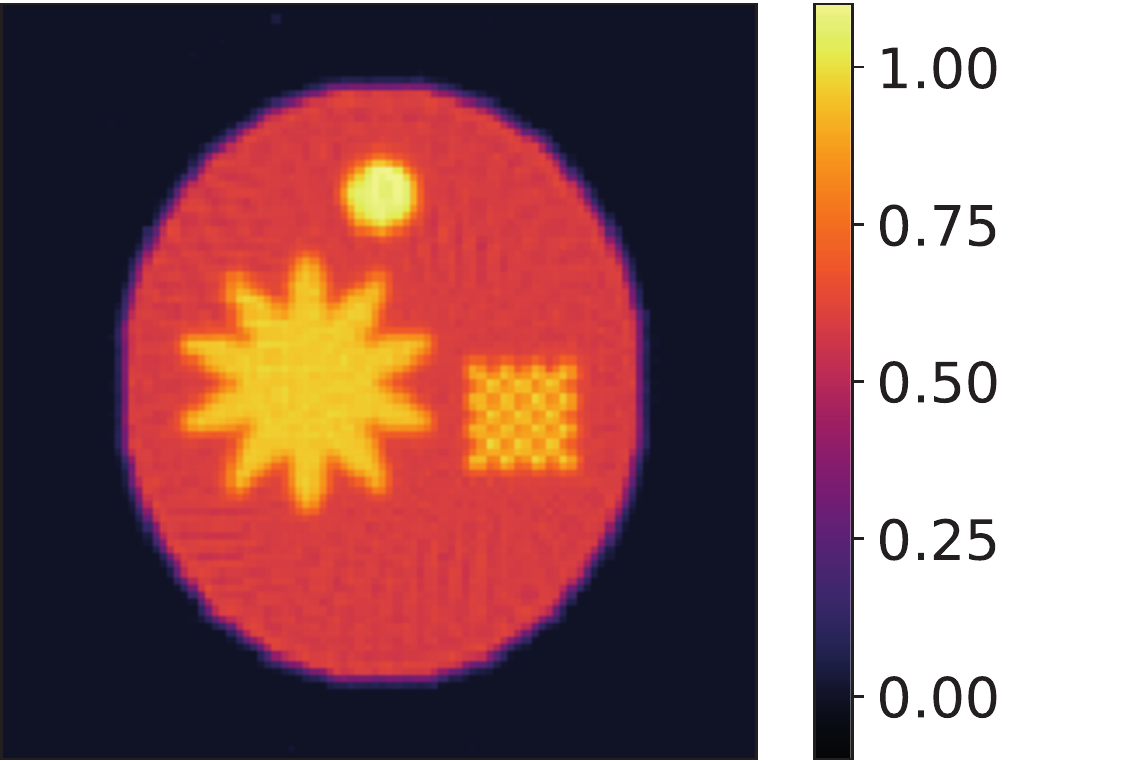}
\includegraphics[width=0.48\textwidth]{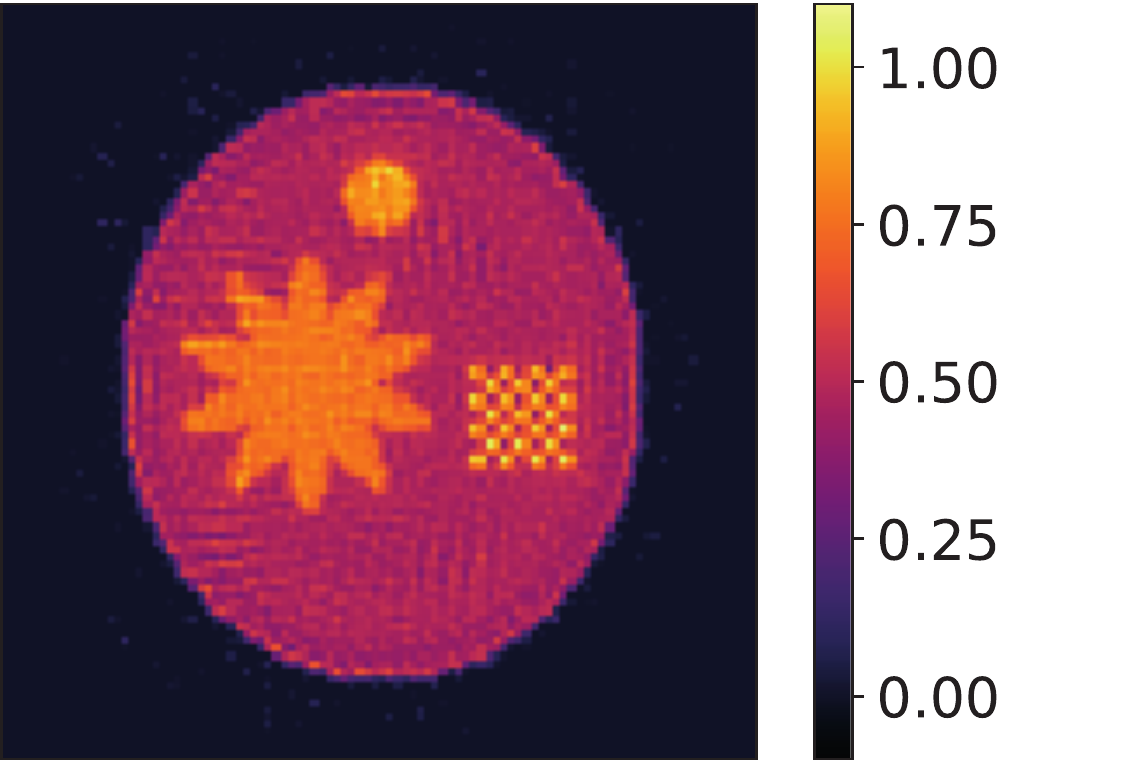}
\\[1em]
\includegraphics[width=0.48\textwidth]{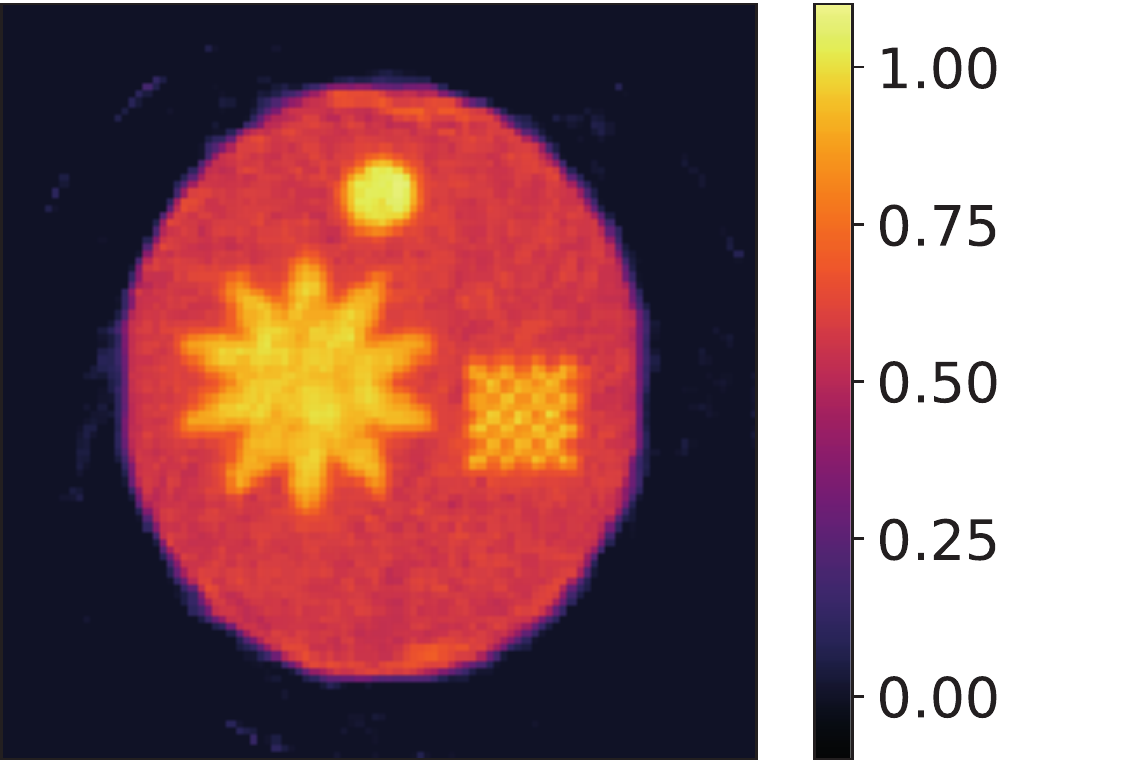}
\includegraphics[width=0.48\textwidth]{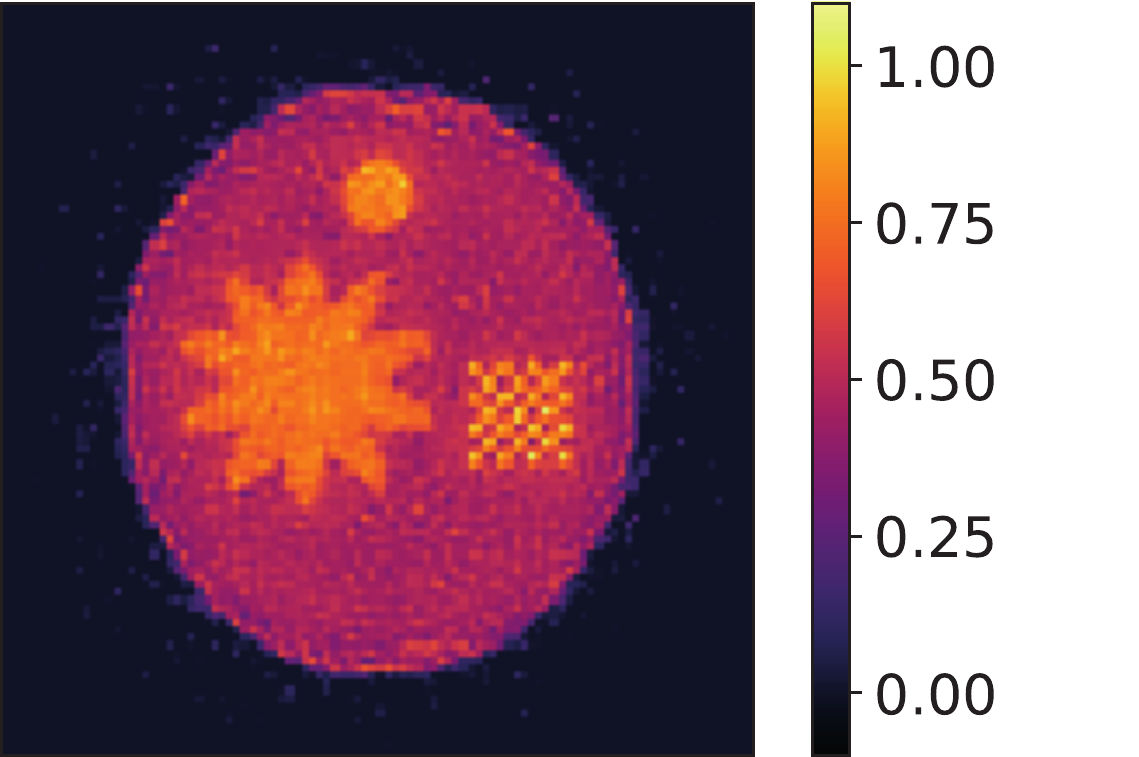}
\caption{Top: Reconstruction from sparse measurements $ \data = \Ao_{m,n} \Wo_n \signal$ with subsampling factor 4 using standard  $\ell^1$-minimization (left) and  the proposed algorithm (right). Bottom: Same for Gaussian measurements.}\label{fig:rec}
\end{figure}

The reconstructions of the convolved phantoms from the undersampled measurements are shown in the right column in Figure~\ref{fig:rec-sparse} and those for the Gaussian measurements are shown in the right column in Figure~\ref{fig:rec-gauss}. The right column in Figure~\ref{fig:rec} shows the resulting reconstructions from undersampled measurements and Gaussian measurements. The left column in Figure~\ref{fig:rec} shows the corresponding reconstructions using standard $\ell^1$-minimization using iterative soft thresholding without the proposed multiscale sparsifying transforms. The relative $\ell^2$ reconstruction errors are $0.17$ (sparse sampling) and $0.19$ (Gaussian measurements) for the proposed method and $0.22$ (both cases) for the standard $\ell^1$-minimization. It is worth noting that the high-resolution pattern is significantly better reconstructed for the proposed multiscale approach than for the standard $\ell^1$-minimization.


\section{Conclusion}
\label{sec:conclusion}

In this paper, we have derived  multiscale factorizations of the wave equation. We have applied the multiscale factorization to CSPAT, where reconstructions are obtained from only a few compressed-sensing measurements consisting of linear combinations of signals recorded by individual detectors. We have presented a novel multiscale reconstruction approach that utilizes the acoustic reciprocity principle to achieve a multiscale decomposition of the desired output pressure through application of a family of operators acting on acoustic data in the time domain. In this way, sparsity of the desired initial pressure distribution is introduced for the high-frequency scales. Our numerical results show that the proposed method improves the reconstructions in the case of compressed sensing measurements.

In future work, we will improve and  analyze the reconstruction algorithm associated to the multiscale factorization. In particular, we analyze the  theoretical  conditions for unique recovery. Other interesting lines of future research is  the extension of the proposed method to PAT with variable sound speed as well as  other tomographic image reconstruction modalities. During  finalization of the  manuscript we found that in the context of Radon inversion with filtered backprojection, related multiscale factorizations  have been proposed in  \cite{peyrin1992multiscale,costin20112d}. The combination of such results with compressed sensing and more advanced reconstruction  techniques seems an interesting line of research.

In future work, we will improve and analyze the reconstruction algorithm associated with multiscale factorization. In particular, we will analyze the theoretical conditions for unique and stable recovery. Other interesting directions of future research is the extension of the proposed method to PAT with variable speed of sound as well as to other tomographic image reconstruction modalities. During the finalization of the manuscript, we found that related multiscale factorizations have been proposed in the context of Radon inversion with filtered backprojection (\cite{peyrin1992multiscale,costin20112d}). Combining such results with compressed sensing and more advanced reconstruction techniques seems to be an interesting direction of further research.

\appendix

\section{Proof of Proposition~\ref{prop:reciD}}
\label{app:reci}

According to   Proposition~\ref{lem:reci} it is sufficient   to show that $\Ro^\sharp  \ky$ is a solution of the equation $\ky = \Ro \Kx$.
Recalling the   definition of $\Ro$ in \eqref{eq:rad} this amounts in showing that $\Ro^\sharp  \ky
= \kx \circ \enorm{} $  satisfies the integral equation  $ \ky(t)   =  \omega_{d-2}  \int_{|t|}^\infty  s \kx(s)  (s^2 - t^2)^{(d-3)/2}  \dd s$ for $t \in \R$. We note that  $(\theta, t) \mapsto \Ro \Kx (t)$ is the Radon transform of the radially symmetric  function $\Kx$. Therefore,  there is  exactly one radial function satisfying the above integral equation.  An explicit expression for the solution has been given in \cite{deans1979gegenbauer}. Using   elementary computation,  a formula has been derived in \cite[p. 23]{natterer1986computerized}. By slight modification we obtain the following results.

 \begin{lemma}\label{lem:natterer}
The  solution $\Kx = \kx  \circ \enorm{} $  of the equation $\ky = \Ro \Kx$ is given by
\begin{equation}\label{eq:natterer}
 \forall r > 0 \colon \quad  \kx(r) \coloneqq   \frac{2 \, (-1)^{d-1}  }{\pi^{(d-1)/2} \Gamma ((d-1)/2)}
\, \Do_r^{d-1}  \int_{r}^{\infty}   (t^2 - r^2)^{(d-3)/2}  \ky (t) \, t \,\dd t \,.
\end{equation}
\end{lemma}

\begin{proof}
In \cite[p. 23]{natterer1986computerized} the identity
 $\Do_r^{d-1} \int_{r}^{\infty}   (t^2 - r^2)^{(d-3)/2}  \ky(t) t \dd t=2^{-1} (-1)^{d-1}   \omega_{d-2}  c(d)  (d-2)!   \kx(r) $ has been derived with
 $c(d) \coloneqq 2^{1-d} \int_{-1}^{1}   (1 - s^2)^{(d-3)/2}   \dd s $. Together with the  identities
 $\omega_{d-2}  = 2 \pi^{(d-1)/2} / \Gamma((d-1)/2) $ and  $c(d) = 2^{1-d} \pi^{1/2} \Gamma  \kl{(d-1)/2}/\Gamma(d/2) $  as well as $\Gamma\kl{d/2}\Gamma  \kl{(d-1)/2} =  2^{2-d} \pi^{1/2} \, (d-2)!$ this yields the explicit solution formula  \eqref{eq:natterer}.
\end{proof}

It remains to bring the right hand side of Equation~\eqref{eq:natterer} into the desired form. We do this separately for the even-dimensional and the odd-dimensional case.
 \begin{itemize}
\item If $d$ is odd, we have
\begin{align*}
    &   \Do_r^{(d-1)/2} \, \Do_r  \, \Do_r^{(d-3)/2}
    \int_{r}^{\infty}  (t^2 - r^2)^{(d-3)/2}  \ky (t) \, t \, \dd t
\\
  &  \hspace{0.1\textwidth} = (-1)^{(d-3)/2}  \left(  (d-3)/2  \right) !  \, \Do_r^{(d-1)/2}  \, \Do_r
    \int_{r}^{\infty}      \ky(t) \, t\, \dd t
\\
  &  \hspace{0.1\textwidth} = (-1)^{(d-1)/2}  \frac{\left(  (d-3)/2  \right) !}{2}  \, \Do_r^{(d-1)/2}  \, \ky(r)    \,.
\end{align*}
Together with Lemma \ref{lem:natterer}, this gives \eqref{eq:reciD}  for $d$ odd.

\item
If $d$ is even, we first compute
\begin{align}
    & \Do_r \nonumber
    \int_{r}^{\infty}   \frac{\ky(t)}{\sqrt{t^2 - r^2}}   t  \dd t
   \\ \nonumber
   & \hspace{0.1\textwidth}  = \Do_r
    \int_{r}^{\infty} \Bigl( \partial_t  \sqrt{t^2 - r^2} \Bigr)   \phi(t)   \dd t
    \\ \nonumber
    & \hspace{0.1\textwidth}  = -  \Do_r
    \int_{r}^{\infty}   \sqrt{t^2 - r^2} \, \Bigl(  \partial_t  \phi(t) \Bigr)    \dd t
    \\  \nonumber
    &
    \hspace{0.1\textwidth}  =
    \int_{r}^{\infty}  \frac{1}{2}       \frac{1}{\sqrt{t^2 - r^2}}  ( \partial_s  \phi(t) )   \dd t
     \\ \label{eq:comm}
     &
     \hspace{0.1\textwidth}  =
     \int_{r}^{\infty}   \frac{\Do_t \phi(t)}{\sqrt{t^2 - r^2}} t   \dd t    \,.
\end{align}
Therefore
\begin{align*}
    &   \Do_r^{d/2}  \,  \Do_r^{(d-2)/2}
     \int_{r}^{\infty}  (t^2 - r^2)^{(d-3)/2}  \ky (t) \, t \, \dd t
\\
  &  \hspace{0.1\textwidth}= (-1)^{(d-2)/2}  \Gamma((d-1)/2) \, \Do_r^{d/2}
    \int_{r}^{\infty}    (t^2 - r^2)^{-1/2}  \, \ky(t) \, t \, \dd t
        \\
     &  \hspace{0.1\textwidth}=   (-1)^{(d-2)/2}   \Gamma((d-1)/2) \,
    \int_{r}^{\infty}   \frac{ \kl{\frac{1}{2t}  \frac{\partial}{\partial t}}^{d/2} \ky(t)  }{\sqrt{t^2 - r^2}}    \,   t \, \dd t \,,
\end{align*}
where the last equality  follows after  $(d/2)$-times   applying equality  \eqref{eq:comm}.
This gives \eqref{eq:reciD}  for $d$ even.
 \end{itemize}

\end{document}